\documentclass[10pt, a4paper]{amsart}
\usepackage[arXiv]{rseaineng}
\usepackage[cal]{rseaincmds}

\allowdisplaybreaks

\usepackage{arydshln}

\newcommand{\tenchi}{{\,}^{t}\!}
\newcommand{\varpivee}{\varpi^\vee}

\DeclareMathOperator{\stt}{R}
\newcommand{\st}{{\stt}}
\DeclareMathOperator{\All}{all}
\newcommand{\all}{{\All}}

\tikzstyle{wv} = [circle, draw, thin, inner sep=1.5pt, minimum size=1.5pt, fill = white]
\tikzstyle{bv} = [circle, draw, thin, inner sep=1.5pt, minimum size=1.5pt, fill = black]
\tikzstyle{vv} = [circle, draw, thin, inner sep=1.5pt, minimum size=1.5pt, fill = white]

\numberwithin{equation}{section}

\title[The characteristic quasi-polynomials of hyperplane arrangements with actions of finite groups]{The characteristic quasi-polynomials of hyperplane arrangements with actions of finite groups}
\author[Ryo Uchiumi]{Ryo Uchiumi}
\date{\today}
\thanks{The author was supported by JSPS KAKENHI, Grant Number 25KJ1735}

\begin{document}

\begin{abstract}
In this paper, we introduce an equivariant version of the characteristic quasi-polynomials as the permutation characters on the complement of mod $q$ hyperplane arrangements.
We prove that the permutation character is a quasi-polynomial in $q$, and show that it can be expressed by the sum of the induced characters of an equivariant version of the Ehrhart quasi-polynomials.
Furthermore, we consider the case of the Coxeter arrangements, and compute in detail for type $A_\ell$.

\end{abstract}

\maketitle

\tableofcontents


\section{Introduction}
\subsection{Quasi-polynomials}
\ \\*
Let $R$ be a commutative ring.
A function $F : \mathbb{Z}_{(>0)} \lra R$ is a \textit{quasi-polynomial} if 
there exist a positive integer $\tilde{n} \in \mathbb{Z}_{>0}$ (called a \textit{period}) and polynomials $f^{(1)},\ldots,f^{(\tilde{n})} \in R[t]$ (called  \textit{constituents}) such that 
\begin{align}
	F(z) = f^{(r)}(z) \IF z \equiv r \pmod{\tilde{n}} \quad (1 \leq r \leq \tilde{n}).
\end{align}
A quasi-polynomial $F$ has the \textit{gcd-property} if $\gcd\{\tilde{n},r_1\} = \gcd\{\tilde{n},r_2\}$ implies $f^{(r_1)} = f^{(r_2)}$, that is, the constituent $f^{(r)}$ depends on $r$ only through $\gcd\{\tilde{n}, r\}$.

Quasi-polynomials are often referred to as ``periodic polynomials".
As will be introduced later, they often appear as counting functions.


Let $L \cong \mathbb{Z}^\ell$ be a lattice of rank $\ell$ and $L_\mathbb{R} \ceq L \otimes \mathbb{R} \cong \mathbb{R}^\ell$.
Let ${P}$ be a rational polytope in $L_\mathbb{R}$.
The \textit{Ehrhart quasi-polynomial} of $P$ is the counting function 
\begin{align}
	\Ell_{P}(q) \ceq \#(q{P} \cap L)
\end{align}
of the lattice points of $q{P}$, which is a quasi-polynomial in $q$ (cf.\! \cite[Theorem 3.23]{BeckRobins}). 
Let ${P}^\circ$ denote the relative interior of ${P}$.
Then the counting function 
\begin{align}
	\Ell_{P}^\circ(q) \ceq \#(q{P}^\circ \cap L)
\end{align}
of the lattice points of $q{P}^\circ$ is also a quasi-polynomial in $q$.
Moreover, the equality
\begin{align}
	\Ell_{P}^\circ(q) = (-1)^{\dim{P}}\Ell_{P}(-q)
\end{align}
is well known as \textit{Ehrhart reciprocity}.

In recent years, Stapledon \cite{Stapledon} proposed an equivariant version of the Ehrhart theory.
Let $\GAMMA$ be a finite group acting linearly on $L$ via $\rho : \GAMMA \lra \GL(L)$.
Suppose that a rational polytope ${P}$ is $\GAMMA$-invariant.
For $q \in \mathbb{Z}_{>0}$, let $\chi_{P,q}$ denote the character of the permutation representation of $\GAMMA$ on $q{P} \cap L$.
Then for each $\gamma \in \GAMMA$, we have
\begin{align}
	\chi_{P,q}(\gamma) = \#(q{P} \cap L)^\gamma = \#\bigset{x \in q{P} \cap L}{\rho(\gamma)(x) = x}.
\end{align}
In particular, $\chi_{P,q}(1)$ is equal to the Ehrhart quasi-polynomial $\Ell_{P}(q)$ as a function in $q$.
Furthermore, Stapledon proved that $\chi_{P,q}$ is a quasi-polynomial in $q$ (whose coefficients contain characters of $\GAMMA$) {\cite[Theorem 5.7]{Stapledon}}.

Let $L^\vee \ceq \Hom_\mathbb{Z}(L,\mathbb{Z})$ denote the dual lattice of $L$.
Given $\alpha_1,\ldots,\alpha_n \in L^\vee$, we can define a (central) hyperplane arrangement $\A = \{H_1,\ldots,H_n\}$ in $L_\mathbb{R}$, with $H_i = \Bigset{x \in L_\mathbb{R}}{\alpha_i(x) = 0}$.

For each positive integer $q \in \mathbb{Z}_{>0}$, define $L_q \ceq L/qL$ and 
\begin{align}
	M(\A; q) \ceq \bigset{\bar{x} \in L_q}{\alpha_i(x) \not\equiv 0 \pmod{q} \tforall i \in \{1,\ldots,n\}}.
\end{align}
It is well known in \cite[Theorem 2.4]{KamiyaTakemuraTerao} that the counting function 
\begin{align}
	\chi_\A^\quasi(q) \ceq \#M(\A; q)
\end{align}
is a quasi-polynomial with gcd-property.
The function $\chi_\A^\quasi$ is called a \textit{characteristic quasi-polynomial} of $\A$, and roughly speaking, it is a mod $q$ version of the Ehrhart quasi-polynomial.
The prime constituent of $\chi_\A^\quasi$ is equal to the \textit{characteristic polynomial} $\chi_\A$ of $\A$ \cite[Theorem 2.2]{Athanasiadis}. The characteristic polynomial $\chi_\A$ is the most important invariant for $\A$, represented by
\begin{align}
	\chi_\A(t) = \sum_{\B \subseteq \A}(-1)^{\#\B}t^{\dim{H_\B}},
\end{align}
where $H_\B = \bigcap_{H \in \B}H$.
Furthermore, for a period $\tilde{n}$ of $\chi_\A^\quasi$, it is known that the $\tilde{n}$-th constituent of $\chi_\A^\quasi$ is equal to the associated toric arrangement (see \cite[Corollary 5.6]{LiuTranYoshinaga}, \cite[Theorem 2.5]{TranYoshinaga}).
Thus, the characteristic quasi-polynomials are also important in terms of toric arrangements and arithmetic matroids (cf. \cite{DAdderioMoci, HigashitaniTranYoshinaga, LiuTranYoshinaga, TranYoshinaga}).

\subsection{Main Results}
\ \\*
In this paper, we propose an equivariant version of the characteristic quasi-polynomial.
Let $\GAMMA$ be a finite group acting linearly on $L$.
Then the actions of $\GAMMA$ on $L_q = L/qL$ and $L^\vee = \Hom_\mathbb{Z}(L,\mathbb{Z})$ are naturally induced, respectively.
Suppose that a hyperplane arrangement $\A$ defined over $L$ is $\GAMMA$-invariant.
We consider the permutation character $\chi_{\A,q}$ of $\GAMMA$ on $M(\A;q)$.
Then 
\begin{align}
	\chi_{\A,q}(\gamma) = \#M(\A;q)^\gamma = \#\bigset{\bar{x} \in M(\A;q)}{\rho_q(\gamma)(\bar{x}) = \bar{x}}
\end{align}
for each $\gamma \in \GAMMA$.
In the case where $\A = \emptyset$, it is proved in \cite[Theorem 1.3]{UchiumiYoshinaga} that $\chi_{\A,q}$ is a quasi-polynomial in $q$ and has the gcd-property.
In this paper, we show that in general.

\begin{theorem}[\cref{Main result 1}] \label{thm1.1}
	The permutation character $\chi_{\A,q}$ is a quasi-polynomial in $q$.
	Furthermore, $\chi_{\A,q}$ has the gcd-property.
\end{theorem}

Let $T \ceq L_\mathbb{R}/L$ be an $\ell$-torus, and let $T(\A)$ be the complement of an arrangement $\A$ in $T$.
For $q \in \mathbb{Z}_{>0}$, let $T[q]$ be the set of the $q$-torsion points in $T$.
Then $T[q]$ can be considered as a substitute $L_q$.
Hence 
\begin{align}
	\chi_\A^\quasi(q) = \#\bigl(T(\A) \cap T[q]\bigr).
\end{align}
Let $\mathcal{C}_T$ denote the set of connected components of $T(\A)$.
Then we have
\begin{align}
	\chi_\A^\quasi(q) = \sum_{C \in \mathcal{C}_T}\Ell_C(q),
\end{align}
where $\Ell_C(q)$ denote the number of the $q$-torsion points in $C$.
We obtain an equivariant version of the above claim.
For $C \in \mathcal{C}_T$, the \textit{$\GAMMA$-orbit} $\GAMMA(C)$ and the \textit{isotropy subgroup} $\GAMMA_{\bar{C}}$ are defined as
\begin{align}
	\GAMMA(C) = \bigset{\gamma(C) \in \mathcal{C}_T}{\gamma \in \GAMMA},\quad \GAMMA_{C} = \bigset{\gamma \in \GAMMA}{\gamma(C) = C}.
\end{align}
For a character $\psi$ of a subgroup $H$ of $\GAMMA$, the \textit{induced character} $\Ind_H^\GAMMA\psi$ of $\psi$ is the character of $\GAMMA$ represented by 
\begin{align}
	\Bigl(\Ind_H^\GAMMA\psi\Bigr)(\gamma) = \dfrac{1}{\#H}\sum_{\substack{\sigma \in \GAMMA\\\sigma^{-1}\gamma\sigma \in H}}\psi(\sigma^{-1}\gamma\sigma).
\end{align}
The main result of this paper is as follows.

\begin{theorem}[\cref{Main result 2}] \label{thm1.2}
	For $q \in \mathbb{Z}_{>0}$, we have
	\begin{align}
		\chi_{\A,q} = \sum_{C \in \mathcal{C}_T}\dfrac{1}{\#\GAMMA(C)}\Ind_{\GAMMA_{C}}^\GAMMA\chi_{C,q} = \sum_{i=1}^k \Ind_{\GAMMA_{C_i}}^\GAMMA\chi_{{C_i},q},
	\end{align}
	where $\{C_1,\ldots,C_k\}$ is the set of all representatives of $\GAMMA$-orbits of $\mathcal{C}_T$.
\end{theorem}

As more detailed examples, we consider an irreducible root system $\PHI$.
Suppose that $L = Z(\PHI)$ is the coweight lattice, $\GAMMA = W(\PHI)$ is the Weyl group, and $\A = \A(\PHI)$ is the Coxeter arrangement.
Let $A_\circ$ be the fundamental alcove of $\PHI$.
Then we obtain the following.
\begin{theorem}[\cref{weylver}] \label{thm1.3}
	For $q \in \mathbb{Z}_{>0}$, we have
	\begin{align}
		\chi_{\A,q} = \Ind^{W(\PHI)}_{W(\PHI)_{{A_\circ}}}\chi_{A_\circ,q}.
	\end{align}
\end{theorem}
By substituting $1$ into the above formula, we recover the equation  obtained in \cite[Proposition 3.7]{YoshinagaW} between the characteristic quasi-polynomial of $\A$ and the Ehrhart quasi-polynomial of $A_\circ$.

The organization of this paper as follows:
To prove \cref{thm1.1}, we review some notions and settings in \cref{sec2.1} and \cref{sec2.2}.
We will prove \cref{thm1.1} in \cref{sec2.3}.
Furthermore, we will consider about a period of $\chi_{\A,q}$, and provide several examples in \cref{sec2.4}.
In \cref{S3}, we consider points of $L_q$ to be $q$-torsion points of the torus $T$.
Moreover, we will give the explicit expression of points fixed by $\gamma \in \GAMMA$.
In \cref{S4}, we will prove \cref{thm1.2}.
In \cref{S5}, we consider the case of the Coxeter arrangements.
We will prove \cref{thm1.3} in \cref{sec5.2}.
Furthermore, we compute $\chi_{\A,q}$ in the case of type $A_\ell$  in \cref{sec5.3}.


\section{Equivariant version of characteristic quasi-polynomials}\label{sec2}

\subsection{Group actions and representations}\label{sec2.1}
\ \\*
We recall notions of representations of finite groups.
For details, refer to \cite{Serre}.

Let $\GAMMA$ be a finite group.
Let $V$ be a finite-dimensional vector space over $\mathbb{C}$.
A \textit{representation} of $\GAMMA$ on $V$ is a group homomorphism $\rho : \GAMMA \lra \GL(V)$, where $\GL(V)$ denote the group of linear isomorphism of $V$ onto itself.
The space $V$ is called the representation space of $\rho$.
The \textit{character} $\chi_\rho : \GAMMA \lra \mathbb{C}$ of the representation $\rho$ is the function defined by $\chi_\rho = \tr \circ \rho$, where $\tr$ is the trace function.
The character $\chi_\rho$ is a class function on $\GAMMA$, which is constant on each conjugacy class of $\GAMMA$.
For characters $\phi,\psi$ of $\GAMMA$, define the inner product $(\phi,\psi)$ as
\begin{align}
	(\phi,\psi) = \dfrac{1}{\#\GAMMA}\sum_{\gamma \in \GAMMA}\phi(\gamma)\overline{\psi(\gamma)},
\end{align}
where $\overline{z}$ denote the complex conjugate of $z \in \mathbb{C}$.
Let $\{\chi_1,\ldots,\chi_k\}$ be the set of all irreducible characters of $\GAMMA$.
Then $\{\chi_1,\ldots,\chi_k\}$ form an orthonormal basis for the space of class functions of $\GAMMA$, that is, $(\chi_i,\chi_j) = \delta_{ij}$, where $\delta_{ij}$ is the Kronecker delta.

Let $H$ be a subgroup of $\GAMMA$, and let $\{g_1,\ldots,g_n\}$ be a set of all representatives in $\GAMMA$ of the left cosets in $G/H$, where $n$ is the index of $H$ in $\GAMMA$.
Let $\rho$ be a representation of $\GAMMA$, and let $\rho_H$ be its restriction to $H$.
Let $W$ be a subrepresentation space of $\rho_H$ and $\theta : H \lra \GL(W)$ this representation.
For $\gamma \in \GAMMA$, the vector space $\rho(\gamma)(W)$ depends only on the left coset of $\gamma$.
Hence the sum $\rho(g_1)(W) + \cdots + \rho(g_n)(W)$ is a subrepresentation space of $V$.
A representation $\rho$ is \textit{induced} by $\theta$ if $V$ is equal to the direct sum $\rho(g_1)(W) \oplus \cdots \oplus \rho(g_n)(W)$.
For any representation $\theta$ of $H$, there exists a representation induced by $\theta$, and it is unique up to isomorphism \cite[Theorem 11]{Serre}.
Such representation is often denoted by $\Ind_H^\GAMMA\theta$.
Let $\chi_\theta$ be the character of $\theta$.
The character of $\Ind_H^\GAMMA\theta$ is denoted by $\Ind_H^\GAMMA\chi_\theta$, and is called the \textit{induced character} of $\chi_\theta$.
It is well known that the following equation holds \cite[Theorem 12]{Serre}:
\begin{align}
	\Bigl(\Ind_H^\GAMMA\chi_\theta\Bigr)(\gamma) = \dfrac{1}{\#H}\sum_{\substack{\sigma \in \GAMMA\\ \sigma^{-1}\gamma\sigma \in H}}\chi_\theta(\sigma^{-1}\gamma\sigma).
\end{align}

Suppose that $\GAMMA$ acts on a finite set $X$.
Let $\mathbb{C}X$ denote the vector space based on $X$.
The $\GAMMA$-action on $X$ induces a natural representation $\rho_X : \GAMMA \lra \GL(\mathbb{C}X)$, which is called the \textit{permutation representation} on $X$.
In particular, $\GAMMA$ acts on itself by the left multiplication.
Its permutation representation is called the \textit{regular representation}.
Since the representation matrix of each $\rho_X(\gamma)$ is a permutation matrix, the character $\chi_X$ of $\rho_X$ is obtained as a counting function of the number of points in $X$ fixed by $\gamma$:
\begin{align}
	\chi_X(\gamma) = \#X^\gamma = \#\bigset{x \in X}{\gamma x = x}.
\end{align}
Furthermore, the character $\chi_\st$ of the regular representation (called the \textit{regular character} of $\GAMMA$) satisfies 
\begin{align}
	\chi_\st(\gamma) = \begin{cases*}
		\#\GAMMA & if $\gamma = 1$;\\
		0 & otherwise.
	\end{cases*}
\end{align}

In this paper, we will use some special characters.
Define $\boldsymbol{1} : \GAMMA \lra \mathbb{C}$ by $\boldsymbol{1}(\gamma) = 1$ for any $\gamma \in \GAMMA$.
Then $\boldsymbol{1}$ is one of the irreducible character of $\GAMMA$ and called the \textit{trivial character}.
Suppose that $\GAMMA$ can be regarded as the subgroup of the general linear group  $\GL_n(\mathbb{Z})$ over integers.
Define $\boldsymbol{\delta} : \GAMMA\lra \mathbb{C}$ by $\boldsymbol{\delta}(\gamma) = \det{\gamma}$.
Then $\boldsymbol{\delta}$ is also one of the irreducible character of $\GAMMA$.

\subsection{Setting}\label{sec2.2}
\ \\*
Let $L$ be a lattice generated by $B = \{\beta_1,\ldots,\beta_\ell\} \subset L$, that is, $L = \mathbb{Z}\beta_1 \oplus \cdots \oplus \mathbb{Z}\beta_\ell$.
We identify an element $x = x_1\beta_1 + \cdots + x_\ell\beta_\ell \in L$ with $(x_1,\ldots,x_\ell) \in \mathbb{Z}^\ell$ naturally.
Let $L^\vee \ceq \Hom_\mathbb{Z}(L,\mathbb{Z})$ denote the dual lattice and $B^\vee = \{\beta_1^\vee,\ldots,\beta_\ell^\vee\}$ be the dual lattice of $B$.
Then $L^\vee$ is generated by $B^\vee$ and we have $\beta_i^\vee(\beta_j) = \delta_{ij}$.
Let $L_\mathbb{R} \ceq L \otimes \mathbb{R}$.
Given $\alpha_1,\ldots,\alpha_n \in L^\vee$, define a central hyperplane arrangement $\A = \{H_1,\ldots,H_n\}$ in $L_\mathbb{R}$ with $H_i = \Bigset{x \in L_\mathbb{R}}{\alpha_i(x) = 0}$.
For each $i \in \{1,\ldots,n\}$, let $s_i = (s_{1i},\ldots,s_{\ell i})$ satisfy $\alpha_i = s_{1i}\beta_1^\vee + \cdots + s_{\ell i}\beta_\ell^\vee$.
Then we have $x \tenchi s_i = \alpha_i(x)$ for $x \in L_\mathbb{R}$.

For a positive integer $q \in \mathbb{Z}_{>0}$, define $L_q \ceq L/qL$ and $\mathbb{Z}_q \ceq \mathbb{Z}/q\mathbb{Z}$, and set
\begin{align}
	M(\A;q) &\ceq \bigset{\bar{x} \in L_q}{\alpha_i(x) \not\equiv 0 \pmod{q} \tforall i \in \{1,\ldots,n\}}\\
	&= \bigset{\bar{x} \in L_q}{\bar{x}[S]_q \in (\mathbb{Z}_q\setminus\{0\})^n},
\end{align}
where $\bar{x} = \pi_q(x)$ for $x \in L$ and the natural projection $\pi_q : L \lra L_q$, and $S = \bigl(s_1 \cdots s_\ell\bigr) = \bigl(s_{ij}\bigr)_{ij}$ is an $\ell \times n$ matrix and $[S]_q$ is the $q$-reduction of $S$.
Note that the set $M(\A; q)$ can be considered as the complement of ``hyperplane arrangement" in $L_q$ as follows:
For $q \in \mathbb{Z}_{>0}$, the set 
\begin{align}
	H_{i,q} \ceq \bigset{\bar{x} \in L_q}{\alpha_i(x) \equiv 0 \pmod{q}}
\end{align}
can be called a ``hyperplane" in $L_q$ and the set $\A_q \ceq \{H_{1,q},\ldots,H_{n,q}\}$ can be called a ``hyperplane arrangement" in $L_q$.
Then we have 
\begin{align}
	M(\A; q) = L_q \setminus \bigcup_{H_{i,q} \in \A_q}H_{i,q}.
\end{align}

The counting function $\chi_\A^\quasi(q) \ceq \#M(\A;q)$ is called a \textit{characteristic quasi-polynomial} of $\A$, introduced by Kamiya--Takemura--Terao in \cite{KamiyaTakemuraTerao}.

\begin{theorem}[{\cite[Theorem 2.4]{KamiyaTakemuraTerao}}]
	Let $\A$ be a central arrangement.
	Then the following statements hold:
	\begin{itemize}
		\item The function $\chi_\A^\quasi$ is a quasi-polynomial with gcd-property.
		\item Each constituent of $\chi_\A^\quasi$ is a monic polynomial of degree $\ell$.
	\end{itemize}
\end{theorem}

For a subset $J \subseteq \{1,\ldots,n\}$, let $S_J$ be the submatrix of $S$ corresponding to $J$.
Define 
\begin{align}
	\tilde{n}_\A \ceq \lcm\bigset{e(J)}{J \subseteq \{1,\ldots,n\},\ J \neq \emptyset} = \lcm\bigset{e(J)}{J \subseteq \{1,\ldots,n\},\ 1 \leq \#J \leq \min\{\ell,n\}},
\end{align}
where $e(J)$ is the maximum elementary divisor of $S_J$.

\begin{proposition}[Higashitani--Tran--Yoshinaga {\cite[Theorem 1.2]{HigashitaniTranYoshinaga}}]\label{HigashitaniTranYoshinaga}
The quasi-polynomial $\chi_\A^\quasi$ has the minimum period $\tilde{n}_\A$ (Such a period is called a $\mathrm{lcm}$-$\mathrm{period}$).
\end{proposition}

Let $\GAMMA$ be a finite group acting linearly on $L$.
In other words, the group homomorphism $\rho : \GAMMA \lra \GL(L)$ is given.
In this paper, suppose that $\rho$ is injective.
Then the $\GAMMA$-action $\rho$ induces $\GAMMA$-actions $\rho_q$ on $L_q$ and $\rho^\vee$ on $L^\vee$, respectively, satisfying
\begin{align}
	\rho_q(\gamma) \circ \pi_q = \pi_q \circ\rho(\gamma),\qquad \rho^\vee(\gamma)(\alpha) = \alpha \circ \rho(\gamma)^{-1}
\end{align}
for $\gamma \in \GAMMA$ and $\alpha \in L^\vee$, where $\pi_q : L \lra L_q$ is the natural projection.

A hyperplane arrangement $\A$ defined over $L$ is \textit{$\GAMMA$-invariant} if $\rho(\gamma)(H_i) \in \A$ for any $\gamma \in \GAMMA$ and $H_i \in \A$, that is, $\rho^\vee(\gamma)\bigl(\{\pm \alpha_1,\ldots,\pm\alpha_n\}\bigr) = \{\pm \alpha_1,\ldots,\pm \alpha_n\}$ for any $\gamma \in \GAMMA$.

\begin{lemma}
	For each $q \in \mathbb{Z}_{>0}$, the complement $M(\A;q)$ is invariant under $\GAMMA$.
\end{lemma}
\begin{proof}
	Let $\bar{x} \in M(\A;q)$.
	Then, $\alpha_i(x) \not\equiv 0 \pmod{q}$ for $i \in \{1,\ldots,n\}$.
	For any $i \in \{1,\ldots,n\}$ and $\gamma \in \GAMMA$, there exists $j \in \{1,\ldots,n\}$ such that $\rho^\vee(\gamma^{-1})(\alpha_i) = \pm\alpha_j$.
	Since
	\begin{align}
		\alpha_i\bigl(\rho(\gamma)(x)\bigr) = \bigl(\rho^\vee(\gamma^{-1})(\alpha_i)\bigr)(x) = \pm\alpha_j(x),
	\end{align}
	we have $\alpha_i\bigl(\rho(\gamma)(x)\bigr) \not \equiv 0 \pmod{q}$.
	Hence $\rho_q(\gamma)(\bar{x}) \in M(\A;q)$.
\end{proof}

Let $\chi_{\A,q}$ denote the permutation representation of $\GAMMA$ on $M(\A;q)$.
For $\gamma \in\GAMMA$, we have 
\begin{align}
	\chi_{\A,q}(\gamma) = \#M(\A;q)^\gamma = \#\bigset{\bar{x} \in M(\A;q)}{\rho_q(\gamma)(\bar{x}) = \bar{x}}.
\end{align}
In particular, $\chi_{\A,q}(1) = \chi_\A^\quasi(q)$, where $1$ is the identity element of $\GAMMA$.
In this sense, the permutation representation $\chi_{\A,q}$ is a generalization of the characteristic quasi-polynomial of $\A$.

\subsection{Quasi-polynomiality}\label{sec2.3}
\ \\*
In this section, we prove that the permutation character $\chi_{\A,q}$ is a quasi-polynomial in $q$.
First, we give the following result on quasi-polynomials.

\begin{lemma}\label{lemer}
	Let $f_1,\ldots,f_k$ be quasi-polynomials and $\tilde{n}_i$ be a period of $f_i$.
	Then the sum $F \ceq f_1 + \cdots + f_k$ is a quasi-polynomial with a period $\lcm\{\tilde{n}_1,\ldots,\tilde{n}_k\}$.
	Furthermore, if $f_1,\ldots,f_k$ have the gcd-property, then $F$ also has the gcd-property.
\end{lemma}
\begin{proof}
	Let $\tilde{n} \ceq \lcm\{\tilde{n}_1,\ldots,\tilde{n}_k\}$.
	Let $f_j^{(r)}$ denote the $r$-th constituent of $f_j$.
	Then, for $r \in \{1,\ldots,\tilde{n}\}$,
	\begin{align}
		F^{(r)} = f_1^{(r_1)} + \cdots + f_k^{(r_k)}
	\end{align}
	becomes the $r$-th constituent of $F$,
	where $r_j$ satisfies $1 \leq r_j \leq \tilde{n}_j$ and $r_j \equiv r \pmod{\tilde{n}_j}$ for each $j \in \{1,\ldots,k\}$.
\end{proof}

For $\gamma \in \GAMMA$, let $R_\gamma$ be the representation matrix of $\rho(\gamma)$ with respect to $B$, that is,
\begin{align}
	\rho(\gamma)(x) = xR_\gamma
\end{align}
for $x \in L$.
Then we have 
\begin{align}
	\chi_{\A,q}(\gamma) = \#\bigset{\bar{x} \in M(\A;q)}{\bar{x}[R_\gamma-I]_q = 0}, \label{pc}
\end{align}
where $I$ is the identity matrix and $[R_\gamma-I]_q$ is the $q$-reduction of the matrix $R_\gamma-I$.

In general, let $A$ be an $\ell \times m$ integer matrix and $d_1,\ldots,d_r$ denote the elementary divisors of $A$, where $d_1 \mid d_2 \mid \cdots \mid d_r$ and $r = \rank{A}$.
Then it is known in {\cite[Lemma 2.1]{KamiyaTakemuraTerao}} and {\cite[Lemma 2.1]{UchiumiYoshinaga}} that the kernel of the $q$-reduction $[A]_q$ is a quasi-monomial in $q$ with the gcd-property and the minimum period $d_r$.
More explicitly, we have
\begin{align}
	\#\bigset{z \in \mathbb{Z}^\ell_q}{z[A]_q = 0} = \Pi(A;q) \cdot q^{\ell-r},
\end{align}
where $\Pi(A;q)$ is a periodic function defined by
\begin{align}
	\Pi(A;q) \ceq \prod_{j=1}^r\gcd\{d_j,q\}.
\end{align}

For $\gamma \in \GAMMA \setminus \{1\}$, let $d(\gamma)$ denote the maximum elementary divisor of $R_\gamma-I$.
Define 
\begin{align}
	\tilde{n}_\GAMMA \ceq \lcm\bigset{d(\gamma)}{\gamma \in \GAMMA \setminus \{1\}}.
\end{align}

Now, we prove that the permutation character $\chi_{\A,q}$ is a quasi-polynomial in $q$.
When $\A$ is an empty arrangement, the author and Yoshinaga have already proved that in \cite{UchiumiYoshinaga}.

\begin{theorem}[{\cite[Theorem 1.3]{UchiumiYoshinaga}}]\label{UchiumiYoshinaga}
	The following statements hold:
	\begin{itemize}
		\item The permutation character $\chi_{\emptyset,q}$ is a quasi-polynomial in $q$.
		\item $\chi_{\emptyset,q}$ has the gcd-property.
		\item The minimum period of $\chi_{\emptyset,q}$ is $\tilde{n}_\GAMMA$.
	\end{itemize}
\end{theorem}

Let $\A$ be an arbitrary $\GAMMA$-invariant arrangement defined over $L$.
The following theorem is one of the main results in this paper.

\begin{theorem}[Restatement of \cref{thm1.1}]\label{Main result 1}
	The permutation character $\chi_{\A,q}$ is a quasi-polynomial in $q$.
	Furthermore, $\chi_{\A,q}$ has the gcd-property.
\end{theorem}
\begin{proof}
	We prove that $\chi_{\A,q}(\gamma)$ is a quasi-polynomial in $q$ and it has gcd-property for each $\gamma \in \GAMMA$.
	Then, for each irreducible character $\chi_i$,
	\begin{align}
		(\chi_i,\,\chi_{\A,{L_q}}) = \dfrac{1}{\#\GAMMA}\sum_{\gamma \in \GAMMA}\chi_i(\gamma)\chi_{\A,q}(\gamma)
	\end{align}
	is also a quasi-polynomial with gcd-property.
	Hence we can see that
	\begin{align}
		\chi_{\A,q} = (\chi_1,\,\chi_{\A,{L_q}})\chi_1 + \cdots + (\chi_k,\,\chi_{\A,{L_q}})\chi_k
	\end{align}
	is a quasi-polynomial with gcd-property.
	
	The equation \cref{pc} implies that 
	\begin{align}
		\chi_{\A,q}(\gamma) &= \#\bigset{\bar{x} \in M(\A;q)}{\bar{x}[R_\gamma - I]_q = 0}\\
		&= \#\bigset{\bar{x} \in L_q}{x(R_\gamma -I) \equiv 0,\ \ \alpha_i(x) \not\equiv 0 \pmod{q} \tforall i \in \{1,\ldots,n\}}\\
		&= \#\bigset{\bar{x} \in L_q}{x(R_\gamma -I) \equiv 0,\ \ x \tenchi s_i \not\equiv 0 \pmod{q} \tforall i \in \{1,\ldots,n\}}.
	\end{align}
	By the inclusion-exclusion principle, we have
	\begin{align}
		\chi_{\A,q}(\gamma) &= \#\bigset{\bar{x}\in L_q}{x(R_\gamma -I) \equiv 0,\ \ x \tenchi s_i \not\equiv 0 \pmod{q} \tforall i \in \{1,\ldots,n\}}\\
		&= \sum_{J \subseteq \{1,\ldots,n\}}(-1)^{\#J}\#\bigset{\bar{x} \in L_q}{x(R_\gamma -I) \equiv 0,\ \ x \tenchi s_j \equiv 0 \pmod{q} \tforall j \in J}\\
		&= \sum_{J \subseteq \{1,\ldots,n\}}(-1)^{\#J}\#\bigset{\bar{x} \in L_q}{\bar{x}[S_{\gamma,J}]_q = 0},
	\end{align}
	where $S_{\gamma,J} = \begin{pmatrix}
		R_\gamma - I & S_J
	\end{pmatrix}$ is a matrix partitioned into two matrices $R_\gamma-I$ and $S_J$.
	Moreover, using the theory of elementary divisors, the above equation can be represented as 
	\begin{align}
		\chi_{\A,q}(\gamma)
		= \sum_{J \subseteq \{1,\ldots,n\}}(-1)^{\#J}\Pi(S_{\gamma,J};q)  q^{\ell-r(\gamma,J)}, \label{sumqp}
	\end{align}
	where $r(\gamma,J) \ceq \rank{S_{\gamma,J}}$.
	Since the equation \cref{sumqp} is a sum of quasi-polynomials with gcd-property, then $\chi_{\A,q}$ is also a quasi-polynomial with gcd-property.
\end{proof}

\subsection{Periods and Examples}\label{sec2.4}
\ \\*
In this section, we consider a period of $\chi_{\A,q}$.
Let $d(\gamma,J)$ denote the maximum elementary divisor of $S_{\gamma,J} = \begin{pmatrix}
	R_\gamma - I & S_J
\end{pmatrix}$ for $\gamma \in \GAMMA$ and $J \subseteq \{1,\ldots,n\}$, where we consider $d(1,\emptyset) = 1$.
From the equation \cref{sumqp} and the elementary result (\cref{lemer}), we can find a period, though not necessarily the minimum period.

\begin{proposition}
	The quasi-polynomial $\chi_{\A,q}$ has a period
	\begin{align}
		\tilde{N} \ceq \lcm\bigset{d(\gamma,J)}{\gamma \in \GAMMA,\ J \subseteq \{1,\ldots,n\}}.
	\end{align}
\end{proposition}

In previous sections, we have defined 
\begin{align}
	\tilde{n}_\A = \lcm\bigset{e(J)}{J \subseteq \{1,\ldots,n\},\ J \neq \emptyset},\quad 
	\tilde{n}_\GAMMA = \lcm\bigset{d(\gamma)}{\gamma \in \GAMMA\setminus\{1\}},
\end{align}
where $e(J)$ (resp. $d(\gamma)$) is the maximum elementary divisor of $S_J$ (resp. $R_\gamma - I$).
Recall that $\tilde{n}_\A$ is the minimum period of the characteristic quasi-polynomial $\chi_\A^\quasi$ (see \cref{HigashitaniTranYoshinaga}) and $\tilde{n}_\GAMMA$ is the minimum period of the permutation character $\chi_{\emptyset,L_q}$ (see \cref{UchiumiYoshinaga}).
Furthermore, $\tilde{n}_\A$ and $\tilde{n}_\GAMMA$ divide $\tilde{N}$, respectively.

\begin{proposition}
	A period of $\chi_{\A,q}$ is divisible by $\tilde{n}_\A$.
\end{proposition}
\begin{proof}
	Let $r_0 \ceq \gcd\{\tilde{n},\tilde{n}_\A\}$.
	Assume that $r_0 \neq \tilde{n}_\A$.
	Let $\tilde{n}$ be a period of $\chi_{\A,q}$, and let $f^{(1)},\ldots,f^{(\tilde{n})}$ be the constituents of $\chi_{\A,q}$.
	Note that $f^{(r)}(q)$ is a function on $\GAMMA$ for each $r \in \{1,\ldots,\tilde{n}\}$ and $q \in \mathbb{Z}_{>0}$.
	Define $g^{(r)}_\gamma(q) \ceq f^{(r)}(q)(\gamma)$.
	Then $g^{(r)}_\gamma$ is considered as a polynomial in $q$ for each $r$.
	Since $\chi_{\A,q}$ has the gcd-property, $\gcd\{\tilde{n},r_1\} = \gcd\{\tilde{n},r_2\}$ implies that $g_\gamma^{(r_1)} = g_\gamma^{(r_2)}$.
	In particular, we have $g_\gamma^{(r_0)} = g_\gamma^{(\tilde{n}_\A)}$.
	For each $q \in \mathbb{Z}_{>0}$, if $\gcd\{\tilde{n},q\} = r$, then 
	\begin{align}
		\chi_\A^\quasi(q) = \chi_{\A,q}(1) = f^{(r)}(q)(1) = g_1^{(r)}(q).
	\end{align}
	Since $\chi_\A^\quasi$ is a quasi-polynomial with the gcd-property and the minimum period $\tilde{n}_\A$, we have $g^{(r_1)}_1 = g^{(r_2)}_1$ if $\gcd\{\tilde{n}_\A,r_1\} = \gcd\{\tilde{n}_\A,r_2\}$.
	Hence $r_0 \neq \tilde{n}_\A$ implies that $g_1^{(r_0)} \neq g_1^{(\tilde{n}_\A)}$.
	This is a contradiction in the fact that $g_1^{(r_0)} = g_1^{(\tilde{n}_\A)}$.	
\end{proof}

It is interesting to see whether a period of $\chi_{\A,q}$ is divisible by $\tilde{n}_\GAMMA$, but the following example shows that it is not generally true.

\begin{example}
	Let  $L \cong \mathbb{Z}^2$ and $\GAMMA = \{1,\gamma\}$ be a group of order $2$ acting linearly on $L$ via $\rho:\GAMMA \lra \GL(L)$ defined by 
	\begin{align}
		\rho(\gamma) : (x_1,x_2) \lmapsto (-x_1,-x_2).
	\end{align}
	Let $\A = \{H_1,H_2,H_3\}$ be a hyperplane arrangement defined by
	\begin{align}
		H_1 = \bigset{(x_1,x_2) \in L_\mathbb{R}}{x_1 = 0},\quad 
		H_2 = \bigset{(x_1,x_2) \in L_\mathbb{R}}{x_2 = 0},\quad 
		H_3 = \bigset{(x_1,x_2) \in L_\mathbb{R}}{x_1 - x_2 = 0}.
	\end{align}
	Then we have
	\begin{align}
		\chi_{L_q}(1) = q^2,\quad 
		\chi_{L_q}(\gamma) = \begin{cases*}
			1 & $\gcd\{2,q\} = 1$;\\
			4 & $\gcd\{2,q\} = 2$,
		\end{cases*}
	\end{align}
	\begin{align}
		\chi_{\A,q}(1) = \chi_\A^\quasi(q) = q^2 - 3q + 2.
	\end{align}
	Hence $\tilde{n}_\A = 1$ and $\tilde{n}_\GAMMA = 2$.
	Since $\gamma$ acts on $L_q$ as the $180^\circ$ rotation, we can see that 
	\begin{align}
		L_q^\gamma = \begin{cases*}
			\bigl\{(0,0)\bigr\} & $\gcd\{2,q\} = 1$;\\
			\bigl\{(0,0),\ (q/2,q/2)\bigr\} & $\gcd\{2,q\} = 2$,
		\end{cases*}
	\end{align} 
	and $L_q^\gamma \subseteq H_{1,q} \cup H_{2,q} \cup H_{3,q}$ for all $q \in \mathbb{Z}_{>0}$.
	Therefore $\chi_{\A,q}(\gamma) = 0$, that is, 
	\begin{align}
		\chi_{\A,q} = \dfrac{q^2-3q+2}{2}\chi_\st = \dfrac{1}{2}(\chi_\st q^2 - 3\chi_\st q + 2\chi_\st).
	\end{align}
	Hence $\chi_{\A,q}$ has the minimum period $1$.
\end{example}

\begin{question}
	Can we characterize hyperplane arrangements $\A$ defined on $L$ invariant under a group $\GAMMA$ such that $\tilde{n}_\GAMMA$ divides a period of $\chi_{\A,q}$?
\end{question}

When $\tilde{n}_\GAMMA$ divides a period of $\chi_{\A,q}$, we expect $\lcm\{\tilde{n}_\A,\tilde{n}_\GAMMA\}$ to be the minimum period of $\chi_{\A,q}$, but this is also generally incorrect.

\begin{example}
	Let  $L \cong \mathbb{Z}^2$ and $\GAMMA = \{1,\gamma\}$ be a group of order $2$ acting linearly on $L$ via $\rho:\GAMMA \lra \GL(L)$ defined by 
	\begin{align}
		\rho(\gamma) : (x_1,x_2) \lmapsto (x_2,x_1).
	\end{align}
	Let $\A = \{H_1\}$ be a hyperplane arrangement defined by
	\begin{align}
		H_1 = \bigset{(x_1,x_2) \in L_\mathbb{R}}{x_1 + x_2 = 0}.
	\end{align}
	Then we have
	\begin{align}
		\chi_{L_q}(1) = q^2,\quad 
		\chi_{L_q}(\gamma) = q,\quad 
	\end{align}
	\begin{align}
		\chi_{\A,q}(1) = \chi_\A^\quasi(q) = q^2-q.
	\end{align}
	Hence $\tilde{n}_\A = \tilde{n}_\GAMMA = 1$.
	Since $\gamma$ acts on $L_q$ as the reflection in the hyperplane defined by $x_1 - x_2 = 0$, we can see that 
	\begin{align}
		L_q^\gamma = \bigset{(x_1,x_2) \in L_q}{x_1 - x_2 = 0},
	\end{align}
	and
	\begin{align}
		L_q^\gamma \cap H_1 = \begin{cases*}
			\bigl\{(0,0)\bigr\} & $\gcd\{2,q\} = 1$;\\
			\bigl\{(0,0),\ (q/2,q/2)\bigr\} & $\gcd\{2,q\} = 2$.
		\end{cases*}
	\end{align}
	Therefore
	\begin{align}
		 \chi_{\A,q}(\gamma) = \begin{cases*}
			q-1 & $\gcd\{2,q\} = 1$;\\
			q-2 & $\gcd\{2,q\} = 2$,
		\end{cases*}\quad 
		\chi_{\A,q} = \begin{cases*}
			\dfrac{1}{2}(\chi_\st q^2 - 2\boldsymbol{\delta}q + \boldsymbol{1} - \boldsymbol{\delta}) & $\gcd\{2,q\} = 1$;\vspace{1mm}\\
			\dfrac{1}{2}(\chi_\st q^2 - 2\boldsymbol{\delta}q + 2\boldsymbol{1} - 2\boldsymbol{\delta}) & $\gcd\{2,q\} = 2$.
		\end{cases*}
	\end{align}
	Hence $\chi_{\A,q}$ has the minimum period $2$.
\end{example}

\begin{question}
	Can we characterize hyperplane arrangements $\A$ defined on $L$ invariant under a group $\GAMMA$ such that $\chi_{\A,q}$ has the minimum period $\lcm\{\tilde{n}_\A,\tilde{n}_\GAMMA\}$?
	Moreover, is it possible to find an explicit formula for the minimum period?
\end{question}

\begin{example}\label{egperiod10}
	Let  $L \cong \mathbb{Z}^2$ and $\GAMMA$ be a cyclic group of order $4$ generated by $\gamma$.
	Suppose that $\GAMMA$ acts linearly on $L$ via $\rho:\GAMMA \lra \GL(L)$ defined by 
	\begin{align}
		\rho(\gamma) : (x_1,x_2) \lmapsto (x_2,-x_1).
	\end{align}
	Let $\A = \{H_1,H_2\}$ be a hyperplane arrangement defined as follows (see also \cref{fig10} in \cref{S4}):
	\begin{align}
		H_1 = \bigset{(x_1,x_2) \in L_\mathbb{R}}{2x_1 - x_2 = 0},\quad 
		H_2 = \bigset{(x_1,x_2) \in L_\mathbb{R}}{x_1 + 2x_2 = 0}.
	\end{align}
	Then we have
	\begin{align}
		\chi_{L_q}(1) = q^2,\quad \chi_{L_q}(\gamma) = \chi_{L_q}(\gamma^3) = \begin{cases*}
			1 & $\gcd\{2,q\} = 1$;\\
			2 & $\gcd\{2,q\} = 2$,
		\end{cases*}\quad
		\chi_{L_q}(\gamma^2) = \begin{cases*}
			1 & $\gcd\{2,q\} = 1$;\\
			4 & $\gcd\{2,q\} = 2$,
		\end{cases*}
	\end{align}
	\begin{align}
		\chi_{\A,q}(1) = \chi_\A^\quasi(q) = \begin{cases*}
			q^2 - 2q + 1 & $\gcd\{5,q\} = 1$;\\
			q^2 - 2q + 5 & $\gcd\{5,q\} = 5$.
		\end{cases*}
	\end{align}
	Hence $\tilde{n}_\A = 5$ and $\tilde{n}_\GAMMA = 2$.
	Since $\gamma$ acts on $L_q$ as the $90^\circ$ rotation, we can see that
	\begin{align}
		L_q^\gamma &= L_q^{\gamma^3} = \begin{cases*}
			\bigl\{(0,0)\bigr\} & $\gcd\{2,q\} = 1$;\\
			\bigl\{(0,0),\ (q/2,q/2)\bigr\} & $\gcd\{2,q\} = 2$,
		\end{cases*}\\
		L_q^{\gamma^2} &= \begin{cases*}
			\bigl\{(0,0)\bigr\} & $\gcd\{2,q\} = 1$;\\
			\bigl\{(0,0),\ (q/2,0),\ (0,q/2),\ (q/2,q/2)\bigr\} & $\gcd\{2,q\} = 2$.
		\end{cases*}
	\end{align}
	and 
	\begin{align}
		L_q^{\gamma^k}\setminus (H_{1,q} \cup H_{2,q}) = \begin{cases*}
			\emptyset & $\gcd\{2,q\} = 1$;\\
			\bigl\{ (q/2,q/2) \bigr\} & $\gcd\{2,q\} = 2$
		\end{cases*}
	\end{align}
	for $k \in \{1,2,3\}$.
	Therefore 
	\begin{align}
		\chi_{\A,q}(\gamma) = \chi_{\A,q}(\gamma^2) = \chi_{\A,q}(\gamma^3) = \begin{cases*}
			0 & $\gcd\{2,q\} = 1$;\\
			1 & $\gcd\{2,q\} = 2$,
		\end{cases*}
	\end{align}
	and
	\begin{align}
		\chi_{\A,q} = \begin{cases*}
			\dfrac{1}{4}(\chi_\st q^2 - 2\chi_\st q + \chi_\st) & $\gcd\{10,q\} = 1$;\vspace{1mm}\\
			\dfrac{1}{4}(\chi_\st q^2 - 2\chi_\st q + 4\boldsymbol{1}) & $\gcd\{10,q\} = 2$;\vspace{1mm}\\
			\dfrac{1}{4}(\chi_\st q^2 - 2\chi_\st q + 5\chi_\st) & $\gcd\{10,q\} = 5$;\vspace{1mm}\\
			\dfrac{1}{4}(\chi_\st q^2 - 2\chi_\st q + 4\chi_\st + 4\boldsymbol{1}) & $\gcd\{10,q\} = 10$.
		\end{cases*}
	\end{align}
	Hence $\chi_{\A,q}$ has the minimum period $10$.
\end{example}

\section{Via $q$-torsion points of the torus} \label{S3}

In this section, we consider the explicit expression of points fixed by $\gamma \in \GAMMA$ using applying elementary operations to a diagonal matrix.
For this purpose, we will identify the points in $L_q$ with the $q$-torsion points of the torus.
 
Let $L$ be a lattice of rank $\ell$ and define $L_q \ceq L/qL$ for $q \in \mathbb{Z}_{>0}$.
Let $T \ceq L_\mathbb{R}/L$ be a $\ell$-torus, and let $\pi_T : L_\mathbb{R} \lra T$ denote the natural projection.
Instead of $L_q$, we will consider the set $T[q]$ of $q$-torsion points of $T$:
\begin{align}
	T[q] = \bigset{t \in T}{qt = 0} = \bigset{\pi_T(x) \in T}{qx \in L}.
\end{align}
Obviously, there exists a bijection $f : T[q] \lra L_q$ defined by
\begin{align}
	f : \pi_T(x) \longmapsto \pi_q(qx) \quad (x \in L_\mathbb{R}). \label{f}
\end{align}

Let $\GAMMA$ be a finite group acting linearly on $L$ via $\rho : \GAMMA \lra \GL(L)$.
Let $\GL(T)$ be the group of the linear transformations of $T$ defined by
\begin{align}
	\GL(T) = \bigset{g : T \lra T}{\text{there exists $g' \in \GL(L_\mathbb{R})$ such that $g \circ \pi_T = \pi_T \circ g'$}}.
\end{align}
Then the $\GAMMA$-action $\rho_T : \GAMMA \lra \GL(T)$ is induced naturally.
In other words, it satisfies
\begin{align}
	\rho_T(\gamma) \circ \pi_T = \pi_T \circ \rho(\gamma)
\end{align}
for each $\gamma \in \GAMMA$.
\begin{proposition}
	The map $f$ defined by \cref{f} is $\GAMMA$-equivalent.
\end{proposition}
\begin{proof}
	Let $\pi_T(x) \in T[q]$ ($x \in L_\mathbb{R}$).
	Then we have
	\begin{align}
		f\Bigl(\rho_T(\gamma)\bigl(\pi_T(x)\bigr)\Bigr) &= f\Bigl(\pi_T\bigl(\rho(\gamma)(x)\bigr)\Bigr)\\
		&= \pi_q\Bigl(q \cdot \bigl(\rho(\gamma)(x)\bigr)\Bigr)\\
		&= \pi_q\bigl(\rho(\gamma)(qx)\bigr)\\
		&= \rho_q(\gamma)\bigl(\pi_q(qx)\bigr)\\
		&= \rho_q(\gamma)\Bigl(f\bigl(\pi_T(x)\bigr)\Bigr).
	\end{align}
	for each $\gamma \in \GAMMA$.
\end{proof}

Since $f$ is bijective, two $\GAMMA$-sets $T[q]$ and $L_q$ are isomorphic.
Hence $\chi_{\emptyset,q}$ is equal to the permutation character of $\GAMMA$ on $T[q]$, that is, we have
\begin{align}
	\chi_{\emptyset,q}(\gamma) = \#T[q]^\gamma = \#\bigset{t \in T[q]}{\rho_T(\gamma)(t) = t}
\end{align}
for any $\gamma \in \GAMMA$.

One of the purposes of this section is to explicitly represent the set $T[q]^\gamma$ for each $\gamma \in \GAMMA$.
It is sufficient that we consider 
\begin{align}
	\pi_T^{-1}(T^\gamma) = \bigset{x \in L_\mathbb{R}}{\rho(\gamma)(x) - x \in L}.
\end{align}
Note that it is clear that $\ker(\rho(\gamma) - \id) \subseteq \pi_T^{-1}(T^\gamma)$, where $\id$ is the identity map. 

\begin{lemma}\label{inv_gamma lemmam}
	For $\gamma \in \GAMMA$, suppose that unimodular matrices $U,V$ and integers $d_1,\ldots,d_\ell$ satisfy 
	\begin{align}
		U(R_\gamma-I)V = \diag(d_1,\ldots,d_\ell). \label{diag}
	\end{align}
	Then 
	\begin{align}
		\pi_T^{-1}(T^\gamma) = \bigoplus_{\substack{1 \leq i \leq \ell\\ d_i \neq 0}}d_i^{-1}\mathbb{Z}u_i \oplus \bigoplus_{\substack{1 \leq i \leq \ell\\d_i = 0}}\mathbb{R}u_i, \label{inv_gamma}
	\end{align}
	where $u_i = u_{i1}\beta_1  \cdots + u_{i\ell}\beta_\ell \in L$ for $U = (u_{ij})_{ij}$. 
\end{lemma}
\begin{proof}
	Let $D \ceq \diag(d_1,\ldots,d_\ell)$.
	For $i \in \{1,\ldots,\ell\}$, 
	\begin{align}
		u_i(R_\gamma - I) = u_iU^{-1}DV^{-1} = d_ie_iV^{-1},
	\end{align}
	where $e_i$ denotes the unit vector with $1$ in the $i$-th component and $0$ otherwise.
	Therefore, if $d_i = 0$, then $\mathbb{R}u_i \subseteq \pi_T^{-1}(T^\gamma)$.
	
	Suppose that $d_i \neq 0$.
	Since $V$ is unimodular, $e_iV^{-1}$ is an integer vector.
	Assume that $e_iT^{-1} \in c\mathbb{Z}^\ell$ for $c \in \mathbb{Z}_{>0}$.
	Then since 
	\begin{align}
		\pm 1 = \det{V^{-1}} = c \cdot \det \tenchi\begin{pmatrix}
			e_1V^{-1} & \cdots & c^{-1}e_iV^{-1} & \cdots & e_\ell V^{-1}
		\end{pmatrix} \in c\mathbb{Z},
	\end{align}
	we obtain $c = 1$.
	Hence we have $d_{\gamma,i}^{-1}u_i(R_\gamma - I) = e_iV^{-1} \in \mathbb{Z}^\ell$, and $d_i^{-1}u_i \in \pi_T^{-1}(T^\gamma)$ follows.
	
	Conversely, let $x \in \pi_T^{-1}(T^\gamma)$.
	Since $\{u_1,\ldots,u_\ell\}$ forms a basis for $L_\mathbb{R}$, we can write $x = c_1u_1 + \cdots + c_\ell u_\ell$ for some $c_1,\ldots,c_\ell \in \mathbb{R}$.
	Then 
	\begin{align}
		x(R_\gamma - I) = \sum_{i=1}^\ell c_iu_i(R_\gamma - I) = \sum_{\substack{1 \leq i \leq \ell\\ d_i \neq 0}}c_id_ie_iV^{-1} \in \mathbb{Z}^\ell.
	\end{align}
	Therefore we can see that $c_i \in d_i^{-1}\mathbb{Z}$ for any $i$ with $d_i \neq 0$.
	Hence $x$ belongs to the right-hand set of \cref{inv_gamma}.
\end{proof}

\begin{theorem}\label{Tgammam}
	For $\gamma \in \GAMMA$, suppose that unimodular matrices $U,V$ and integers $d_1,\ldots,d_\ell$ satisfy \cref{diag}.
	Then
	\begin{align}
		T^\gamma = \pi_T\left( \bigoplus_{\substack{1 \leq i \leq \ell\\d_i \neq 0}}d_i^{-1}\mathbb{Z}u_i \oplus \bigoplus_{\substack{1 \leq i \leq \ell\\d_i = 0}}\mathbb{R}u_i \right)
		= \Bigset{ \sum_{\substack{1 \leq i \leq \ell\\d_i \neq 0}}\pi_T({a_i \cdot d_i^{-1}u_i}) + \sum_{\substack{1 \leq i \leq \ell\\d_i = 0}}\pi_T({b_i \cdot u_i})}{a_i \in \mathbb{Z},\ b_i \in \mathbb{R}}.
	\end{align}
	Furthermore, for $q \in \mathbb{Z}_{>0}$, we have
	\begin{align}
		T[q]^\gamma
		= \Bigset{ \sum_{\substack{1 \leq i \leq \ell\\d_i \neq 0}}\pi_T\bigl({a_i \cdot g_i(q)^{-1}u_i}\bigr) + \sum_{\substack{1 \leq i \leq \ell\\d_i = 0}}\pi_T({b_i \cdot q^{-1}u_i})}{a_i, b_i \in \mathbb{Z}}, \label{Tqgamma}
	\end{align}
	where $u_i \in L$ is same as given in \cref{inv_gamma lemmam} and $g_i(q) \ceq \gcd\{d_i,q\}$.
\end{theorem}

For $\gamma \in \GAMMA$, let $r(\gamma) \ceq \rank(R_\gamma - I)$ and $d_{\gamma,1},\ldots,d_{\gamma,r(\gamma)}$ denote the elementary divisors of $R_\gamma - I$ satisfying $d_{\gamma,1} \mid d_{\gamma,2} \mid \cdots \mid d_{\gamma,r(\gamma)}$.
Then there exist $\ell\times \ell$ unimodular matrices $U,V$ (not unique) such that 
\begin{align}
	U(R_\gamma - I)V = \diag(d_{\gamma,1},\ldots,d_{\gamma,r(\gamma)},0,\ldots,0).
\end{align}
The right-hand matrix of the above is called a \textit{Smith normal form} of $R_\gamma - I$ 
(In this paper, we call $U$ (resp. $V$) a \textit{left} (resp. \textit{right}) \textit{transformation matrix} of $R_\gamma - I$).

\begin{corollary}\label{Tgamma}
	Let $\gamma \in \GAMMA$.
	Then
	\begin{align}
		T^\gamma = \pi_T\left( \bigoplus_{i \leq r(\gamma)}d_{\gamma,i}^{-1}\mathbb{Z}u_i \oplus \bigoplus_{i > r(\gamma)}\mathbb{R}u_i \right)
		= \Bigset{ \sum_{i \leq r(\gamma)}\pi_T({a_i \cdot d_{\gamma,i}^{-1}u_i}) + \sum_{i > r(\gamma)}\pi_T({b_i \cdot u_i})}{a_i \in \mathbb{Z},\ b_i \in \mathbb{R}}.
	\end{align}
	Furthermore, for $q \in \mathbb{Z}_{>0}$, we have
	\begin{align}
		T[q]^\gamma
		= \Bigset{ \sum_{i \leq r(\gamma)}\pi_T\bigl({a_i \cdot g_{\gamma,i}(q)^{-1}u_i}\bigr) + \sum_{i > r(\gamma)}\pi_T({b_i \cdot q^{-1}u_i})}{a_i, b_i \in \mathbb{Z}},
	\end{align}
	where $u_i = u_{i1}\beta_1 + \cdots + u_{i\ell}\beta_\ell \in L$ for a left transformation matrix $U = (u_{ij})_{ij}$ of $R_\gamma - I$, and $g_{\gamma,i}(q) \ceq \gcd\{d_{\gamma,i},q\}$.
\end{corollary}

In particular, \cref{Tgamma} implies that $T^\gamma$ is a disjoint union of $d_{\gamma,1} \cdots d_{\gamma,r(\gamma)}$ tori of dimension $\ell-r(\gamma)$ (see also \cite[\S6]{HigashitaniTranYoshinaga}).

Let $\A = \{H_1,\ldots,H_n\}$ be a hyperplane arrangement in $L_\mathbb{R}$ defined on $L$.
Fix a linear form $\alpha_i \in L^\vee$ satisfying $H_i = \ker{\alpha_i}$ for each $i \in \{1,\ldots,n\}$.
Define the set
\begin{align}
	\A^\aff \ceq \bigset{H_i^k \ceq \Bigset{x \in L_\mathbb{R}}{\alpha_i(x) = k}}{i \in \{1,\ldots,n\},\ \ k \in \mathbb{Z}}
\end{align}
of affine hyperplanes translating the hyperplanes in $\A$.
Let 
\begin{align}
	M(\A^\aff) = L_\mathbb{R} \setminus \bigcup_{H \in \A^\aff}H
\end{align}
be the complement of affine arrangement $\A^\aff$.

\begin{lemma}\label{l2.10}
	Let $x,y \in L_\mathbb{R}$ satisfying $y-x \in L$.
	If $x \in M(\A^\aff)$, then $y \in M(\A^\aff)$.
	In other words, if $\pi_T(x) = \pi_T(y)$, then $x \in M(\A^\aff)$ implies $y \in M(\A^\aff)$.
\end{lemma}
\begin{proof}
	Let $x,y \in L_\mathbb{R}$ satisfying $y-x \in L$.
	Suppose that $x \in M(\A^\aff)$.
	Then $\alpha_i(x) \not\in \mathbb{Z}$ for any $i \in \{1,\ldots,n\}$.
	Since $\alpha_i(y-x) \in \mathbb{Z}$, we have 
	\begin{align}
		\alpha_i(y) = \alpha_i(y-x) + \alpha_i(x) \not\in \mathbb{Z}
	\end{align}
	for any $i \in \{1,\ldots,n\}$.
	Hence $y \in M(\A^\aff)$.
\end{proof}

We consider the subset
\begin{align}
	T(\A) \ceq \pi_T\bigl(M(\A^\aff)\bigr) = \bigset{\pi_T(x) \in T}{\alpha_i(x) \not \in \mathbb{Z} \tforall i \in \{1,\ldots,n\}}
\end{align}
of the torus $T$.
It can be regarded as the complement of $\A$ in $T$.
That is, each hyperplane in $\A$ forms a subtorus of dimension $\ell-1$ in $T$.
For $q \in \mathbb{Z}_{>0}$, define 
\begin{align}
	T(\A)[q] \ceq T(\A) \cap T[q] = \bigset{\pi_T(x) \in T[q]}{\alpha_i(x) \not \in \mathbb{Z} \tforall i \in \{1,\ldots,n\}}.
\end{align}

\begin{lemma}\label{l2.11}
	Let $f : T[q] \lra L_q$ be the bijection defined by \cref{f}.
	Then $f\bigl(T(\A)[q]\bigr) = M(\A;q)$.
\end{lemma}
\begin{proof}
	Let $\pi_T(x) \in T(\A)[q]$ ($x \in L_\mathbb{R}$).
	Then $qx \in L$.
	For any $i \in \{1,\ldots,n\}$, since $\alpha_i(x) \not\in \mathbb{Z}$, we have $\alpha_i(qx) \in \mathbb{Z}\setminus q\mathbb{Z}$, that is, $\alpha_i(qx) \not\equiv 0 \pmod{q}$.
	Hence $f\bigl(\pi_T(x)\bigr) = \pi_q(qx)$ belongs to $M(\A;q)$.
	
	Conversely, let $\pi_q(y) \in M(\A;q)$ ($y \in L$).
	Since $f$ is bijective and $\pi_T$ is surjective, there exists $x \in L_\mathbb{R}$ such that $f\bigl(\pi_T(x)\bigr) = \pi_q(qx) = \pi_q(y)$.
	For any $i \in \{1,\ldots,n\}$, since $\alpha_i(qx) \not\equiv 0 \pmod{q}$, we have $\alpha_i(x) \not\in \mathbb{Z}$.
	Therefore $\pi_T(x) \in T(\A)[q]$, and hence $\pi_q(y) \in f(T(\A)[q])$.
\end{proof}

Suppose that an arrangement $\A$ is $\GAMMA$-invariant.
Then $T(\A)[q]$ is invariant under $\GAMMA$ and isomorphic to $M(\A;q)$ as a $\GAMMA$-set.

\begin{corollary}
	For $q \in \mathbb{Z}_{>0}$ and $\gamma \in \GAMMA$, we have
	\begin{align}
		\chi_{\A,q}(\gamma) &= \#T(\A)[q]^\gamma = \#\bigset{t \in T(\A)[q]}{\rho_T(\gamma)(t) = t}\\
		&= \#\bigl( T(\A) \cap T[q]^\gamma \bigr).
	\end{align}
\end{corollary}

Here, we show that for fixed $\gamma \in \GAMMA$, $\chi_{\A,q}(\gamma)$ can be expressed using several characteristic quasi-polynomials of (not necessarily central) hyperplane arrangement restricting $\A$.

Fix $\gamma \in \GAMMA$, and suppose again that unimodular matrices $U,V$ and integers $d_1,\ldots,d_\ell$ satisfy
\begin{align}
	U(R_\gamma-I)V = \diag(d_1,\ldots,d_\ell),
\end{align}
and let $u_i$ denote the elements of $L$ corresponding to the $i$-th row vector of $U$.
Define a set
\begin{align}
	\XI \ceq \Bigset{\sum_{\substack{1 \leq i \leq \ell\\d_i \neq 0}}a_i \cdot d_i^{-1}u_i \in L_\mathbb{R}}{a_i \in \{1,\ldots,d_i\}}
\end{align}
and a sublattice 
\begin{align}
	M \ceq \bigoplus_{\substack{1 \leq i \leq \ell\\d_i = 0}}\mathbb{Z}u_i.
\end{align}
of $L$.
Let $T_M \ceq M_\mathbb{R}/M$ be a subtorus of $T$.
Note that the restriction $\pi_T|_{T_M}$ is equal to the natural projection from $M_\mathbb{R}$ onto $T_M$.
For $\xi \in \XI$, define a function $\epsilon_\xi : \mathbb{Z}_{>0} \lra \{0,1\}$ as
\begin{align}
	\epsilon_\xi(q) = \begin{cases*}
		0 & $\pi_T(\xi) \not\in T[q]$;\\
		1 & $\pi_T(\xi) \in T[q]$.
	\end{cases*}
\end{align}
Then, for $q \in \mathbb{Z}_{>0}$, the equation \cref{Tqgamma} implies that 
\begin{align}
	T[q]^\gamma &= \bigset{\pi_T(\xi) + \pi_T(x)}{\xi \in \XI \cap \pi_T^{-1}(T[q]),\ \ \pi_T(x) \in T_M[q]}\\
	&= \bigsqcup_{\substack{\xi \in \XI\\ \epsilon_\xi(q) = 1}}\bigset{\pi_T(\xi) + \pi_T(x)}{\pi_T(x) \in T_M[q]}.
\end{align}
Therefore we have
\begin{align}
	\chi_{\A,q}(q) &= \#\bigl( T(\A) \cap T[q]^\gamma \bigr)\\
	&= \sum_{\xi \in \XI}\epsilon_\xi(q) \cdot \#\bigset{\pi_T(\xi) + \pi_T(x)}{\pi_T(x) \in T_M[q],\ \ \alpha_i(\xi) + \alpha_i(x) \not\in \mathbb{Z} \tforall i \in \{1,\ldots,n\}}\\
	&= \sum_{\xi \in \XI}\epsilon_\xi(q) \cdot \#\bigset{\pi_T(x) \in T_M[q]}{\alpha_i(\xi) + \alpha_i(x) \not\in \mathbb{Z} \tforall i \in \{1,\ldots,n\}}. \label{eqab}
\end{align}
The terms of \eqref{eqab} closely resemble but are not identical to the characteristic quasi-polynomials of (non-central) hyperplane arrangements defined on the lattice $M$ (the characteristic quasi-polynomials in the non-central case are introduced in \cite{KamiyaTakemuraTeraononcentral}).

\section{Via the equivariant Ehrhart theory} \label{S4}

Let $\A = \{H_1,\ldots,H_n\}$ be a (non-empty) hyperplane arrangement in $L_\mathbb{R}$ defined on $L$.
Let $\A^\aff$ denote the set of affine hyperplanes translating the hyperplanes in $\A$.
The connected components of the complement $M(\A^\aff)$ of affine arrangement $\A^\aff$ can be represented as 
\begin{align}
	C = \bigcap_{H_i \in \A}H_i^{k_i,+} \cap H_i^{k_i+1,-}
\end{align}
using open half spaces
\begin{align}
	H_i^{k_i,+} = \bigset{x \in L_\mathbb{R}}{\alpha_i(x) > k_i},\quad H_i^{k_i+1,-} = \bigset{x \in L_\mathbb{R}}{\alpha_i(x) < k_i+1},
\end{align}
where $k_i \in \mathbb{Z}$ for each $H_i \in \A$.
We shall call each component a \textit{chamber} of $\A^\aff$, and denote the set of them by $\mathcal{C}(\A^\aff)$.

For each $i \in \{1,\ldots,\ell\}$, define the map $t_i : L_\mathbb{R} \lra L_\mathbb{R}$ by 
\begin{align}
	t_i : (x_1,\ldots,x_i,\ldots,x_\ell) \lmapsto (x_1,\ldots,x_i+1,\ldots,x_\ell).
\end{align}
For each chamber $C \in \mathcal{C}(\A^\aff)$, it is clear that $t_i(C)$ is congruent to $C$ and also a chamber of $\A^\aff$.
Let $\mathcal{C}_T$ denote the set of connected components of $T(\A)$.
Then we obtain
\begin{align}
	\mathcal{C}_T = \bigset{\pi_T(C)}{C \in \mathcal{C}(\A^\aff)}.
\end{align}
We shall also refer to each component as \textit{chamber}.
Then we can see that
\begin{align}
	\chi_\A^\quasi(q) = \sum_{\pi_T(C) \in \mathcal{C}_T}\Ell_{\pi_T(C)}(q), \label{decompcqp}
\end{align}
where 
\begin{align}
	\Ell_{\pi_T(C)}(q) = \#\bigl(\pi_T(C) \cap T[q]\bigr).
\end{align}
If $C$ is bounded, then $\Ell_{\pi_T(C)}$ coincides with the Ehrhart quasi-polynomial $\Ell_C$.
Therefore the characteristic quasi-polynomial $\chi_\A^\quasi$ can be regarded as a sum of Ehrhart quasi-polynomials (see also \cite[\S2.2]{KamiyaTakemuraTerao}).
For simplicity, we write $\pi_T(C)$ simply as $C$.

Let $\GAMMA$ be a finite group acting linearly on $L$ via $\rho :\GAMMA \lra \GL(L)$.
Suppose that $\A$ is $\GAMMA$-invariant.

\begin{lemma}
	Let $C \in \mathcal{C}(\A^\aff)$ and $\gamma \in \GAMMA$. 
	Then $\rho(\gamma)(C)$ is also a chamber of $\A^\aff$.
	Hence $\GAMMA$ also acts on $\mathcal{C}_T$.
\end{lemma}
\begin{proof}
	Since $\rho(\gamma)(C)$ is a connected subset of $M(\A^\aff)$, there exists a chamber $C' \in \mathcal{C}(\A^\aff)$ such that $\rho(\gamma)(C) \subseteq C'$.
	Similarly, there exists a chamber $C'' \in \mathcal{C}(\A^\aff)$ such that $\rho(\gamma^{-1})(C') \subseteq C''$.
	The fact $C \subseteq \rho(\gamma^{-1})(C')$ implies that $C = C''$.
	Hence we have $C' \subseteq \rho(\gamma)(C)$.
	Thus $\rho(\gamma)(C) = C' \in \mathcal{C}(\A^\aff)$.
\end{proof}

For $C \in \mathcal{C}_T$, let $\GAMMA_{C}$ denote the isotropy subgroup with respect to $C$.
Since $C$ is invariant under $\GAMMA_C$, we can consider the permutation character $\chi_{C,q}$ of $\GAMMA_C$ on $C \cap T[q]$.
Therefore $\chi_{\A,q}(\gamma)$ can be expressed using permutation characters for each $\gamma \in \GAMMA$.

\begin{theorem}
	For each $\gamma \in \GAMMA$, we have
	\begin{align}
		\chi_{\A,q}(\gamma) = \sum_{C \in \mathcal{C}_T^\gamma}\chi_{C,q}(\gamma). \label{decompgamma}
	\end{align}
\end{theorem}
\begin{proof}
	Suppose that $C \not\in \mathcal{C}_T^\gamma$. 
	Then $\rho_T(\gamma)(t) \not\in C$ for any $t \in C$, that is, $\rho_T(\gamma)(t) \neq t$ for any $t \in C$.
	Hence 
	\begin{align}
		\chi_{\A,q}(\gamma) &= \sum_{C \in \mathcal{C}_T}\#\bigset{t \in C \cap T[q]}{\rho_T(\gamma)(t) = t}\\
		&= \sum_{C \in \mathcal{C}_T^\gamma}\#\bigset{t \in C \cap T[q]}{\rho_T(\gamma)(t) = t}\\
		&= \sum_{C \in \mathcal{C}_T^\gamma}\chi_{C,q}(\gamma).
	\end{align}
\end{proof}

By improving the equation \cref{decompgamma} using the induced characters to a form independent of $\gamma \in \GAMMA$, we can obtain an equivariant version of the equation \cref{decompcqp}.
Let $\GAMMA(C)$ be the $\GAMMA$-orbit of $C$ .
The following is the main theorem of this paper.

\begin{theorem}[Restatement of \cref{thm1.2}]\label{Main result 2} 
	Let $\A$ be a non-empty arrangement.
	For $q \in \mathbb{Z}_{>0}$, we have
	\begin{align}
		\chi_{\A,q} = \sum_{C \in \mathcal{C}_T}\dfrac{1}{\#\GAMMA(C)}\Ind^\GAMMA_{\GAMMA_C}\chi_{C,q}
		 = \sum_{i=1}^k\Ind^\GAMMA_{\GAMMA_{C_i}}\chi_{C_i,q},
	\end{align}
	where $\{C_1,\ldots,C_k\}$ is the set of all representatives of $\GAMMA$-orbits of $\mathcal{C}_T$.
\end{theorem}
\begin{proof}
	For each representative $C_i$, we obtain
	\begin{align}
		\Bigl( \Ind^\GAMMA_{\GAMMA_{C_i}}\chi_{C_i,q} \Bigr)(\gamma)
		&= \dfrac{1}{\#\GAMMA_{C_i}}\sum_{\substack{\sigma \in \GAMMA\\ \sigma^{-1}\gamma\sigma \in \GAMMA_{C_i}}}\chi_{C_i,q}(\sigma^{-1}\gamma\sigma)\\
		&= \dfrac{1}{\#\GAMMA_{C_i}}\sum_{\substack{\sigma \in \GAMMA\\ \gamma\sigma(C_i) = \sigma(C_i)}}\#\bigl(C_i \cap T[q]\bigr)^{\sigma^{-1}\gamma\sigma}\\
		&= \dfrac{1}{\#\GAMMA_{C_i}}\sum_{\sigma \in \GAMMA}\#\bigl(\rho_T(\sigma)(C_i) \cap T[q]\bigr)^\gamma\\
		&= \#\left(\bigcup_{\sigma \in \GAMMA}\bigl(\rho_T(\sigma)(C_i) \cap T[q]\bigr)\right)^\gamma.
	\end{align}
	Hence 
	\begin{align}
		\chi_{\A,q}(\gamma) &= \#T(\A)[q]^\gamma\\
		&= \#\left( \bigcup_{C \in\mathcal{C}_T}\bigl(C \cap T[q]\bigr) \right)^\gamma\\
		&= \sum_{i=1}^k \#\left(\bigcup_{\sigma \in \GAMMA}\bigl(\rho_T(\sigma)(C_i) \cap T[q]\bigr)\right)^\gamma\\
		&= \sum_{i=1}^k \Bigl( \Ind^\GAMMA_{\GAMMA_{C_i}}\chi_{C_i,q} \Bigr)(\gamma).
	\end{align}
\end{proof}

\begin{example}\label{egdecomp}
	Let $L \cong \mathbb{Z}^2$ and $\GAMMA$ be a cyclic group of order $4$ generated by $\gamma$.
	Suppose that $\GAMMA$ acts linearly on $L$ via $\rho:\GAMMA \lra \GL(L)$ defined by 
	\begin{align}
		\rho(\gamma) : (x_1,x_2) \lmapsto (x_2,-x_1).
	\end{align}
	Let $\A = \{H_1,H_2\}$ be a hyperplane arrangement defined by
	\begin{align}
		H_1 = \bigset{(x_1,x_2) \in L_\mathbb{R}}{2x_1 - x_2 = 0},\quad 
		H_2 = \bigset{(x_1,x_2) \in L_\mathbb{R}}{x_1 + 2x_2 = 0}.
	\end{align}
	We have already computed $\chi_{\A,q}$ in \cref{egperiod10}.
	Here, we consider the expression of $\chi_{\A,q}$ as a sum of induced representations of permutation representations for chambers.
	As shown in \cref{fig10}, we can see that
	\begin{align}
		\mathcal{C}_T = \bigl\{\pi_T(P),\, \pi_T(C_1),\, \pi_T(C_2),\, \pi_T(C_3),\, \pi_T(C_4)\bigr\},
	\end{align}
	 where $P$ is the square with center $(1/2,1/2)$ and $C_1,C_2,C_3,C_4$ are squares around $P$.

	\begin{figure}[b]
		\begin{tikzpicture}[scale = 0.32]
			\draw[gray, very thin, ->] (-3,0) -- (10+2,0);
			\draw[gray, very thin, ->] (0,-2) -- (0,10+3);
			\draw[gray, dashed, very thin] (10,0) -- (10,10) -- (0,10);
			\draw[semithick] (-1/2,-2/2) node[xshift=-6,yshift=-6]{$H_1^0$} -- (1+5+1/2,2+10+2/2);
			\draw[semithick] (-1-1/2+5,-2-2/2) node[xshift=-6,yshift=-6]{$H_1^1$} -- (5+1/2+5,10+2/2);
			\draw[gray, semithick] (-1/2+8,-4-2/2) node[xshift=-6,yshift=-6]{$H_1^{2}$} -- (4+1/2+10,8+2/2);
			\draw[gray, semithick] (-1/2-4,2-2/2) node[xshift=-4,yshift=-6]{$H_1^{-1}$} -- (5+1/2-5+2,4+10+2/2);
			\draw[semithick] (-2-2/2,1+5+1/2) node[xshift=-6,yshift=6]{$H_2^1$} -- (10+2/2,-1/2);
			\draw[semithick] (-2/2,5+1/2+5) node[xshift=-6,yshift=6]{$H_2^2$} -- (2+10+2/2,-1-1/2+5);
			\draw[gray, semithick] (2-2/2,4+1/2+10) node[xshift=-6,yshift=6]{$H_2^3$} -- (4+10+2/2,-1/2+8);
			\draw[gray, semithick] (-4-2/2,2+5+1/2-5) node[xshift=-6,yshift=6]{$H_2^0$} -- (8+2/2,-1/2-4);
			\draw (0, 0) node[bv](00){};
			\draw (5, 5) node[bv](00){};
			\draw (5, 0) node[wv](00){};
			\draw (0, 5) node[wv](00){};
			\draw (3, 1) node{$C_1$};
			\draw (9, 3) node{$C_2$};
			\draw (7, 9) node{$C_3$};
			\draw (1, 7) node{$C_4$};
			\draw[gray] (3, 1+10) node{$t_2(C_1)$};
			\draw[gray] (9-10, 3) node{$t_1^{-1}(C_2)$};
			\draw[gray] (7, 9-10) node{$t_2^{-1}(C_3)$};
			\draw[gray] (1+10, 7) node{$t_1(C_4)$};
			\draw (4, 6) node{$P$};
		\end{tikzpicture}
		
		\caption{$\A^\aff$ in \cref{egperiod10} and \cref{egdecomp}, the points $\bullet$ are fixed by $\rho_T(\gamma)$, the points $\circ$ are fixed by $\rho_T(\gamma^2)$.}
		\label{fig10}
	\end{figure}

	Each squares have the volume $1/5$, and the Ehrhart quasi-polynomials can be computed as
	\begin{align}
		\Ell_P(q) &= \begin{cases*}
			\dfrac{1}{5}(q^2 - 2q + 1) &  $q \equiv 1 \pmod{5}$;\vspace{1mm}\\
			\dfrac{1}{5}(q^2 - 2q + 5) &  $q \equiv 2 \pmod{5}$;\vspace{1mm}\\
			\dfrac{1}{5}(q^2 - 2q - 3) &  $q \equiv 3 \pmod{5}$;\vspace{1mm}\\
			\dfrac{1}{5}(q^2 - 2q - 3) &  $q \equiv 4 \pmod{5}$;\vspace{1mm}\\
			\dfrac{1}{5}(q^2 - 2q + 5) &  $q \equiv 5 \pmod{5}$,
		\end{cases*}\\[4pt]
		\Ell_{C_i}(q) &= \begin{cases*}
			\dfrac{1}{5}(q^2 - 2q + 1) &  $q \equiv 1 \pmod{5}$;\vspace{1mm}\\
			\dfrac{1}{5}(q^2 - 2q)     &  $q \equiv 2 \pmod{5}$;\vspace{1mm}\\
			\dfrac{1}{5}(q^2 - 2q + 2) &  $q \equiv 3 \pmod{5}$;\vspace{1mm}\\
			\dfrac{1}{5}(q^2 - 2q + 2) &  $q \equiv 4 \pmod{5}$;\vspace{1mm}\\
			\dfrac{1}{5}(q^2 - 2q + 5) &  $q \equiv 5 \pmod{5}$.
		\end{cases*}
	\end{align}	
	We can see that $\rho_T(\gamma)(P) = P$ and $\rho_T(\gamma)(C_i) = C_{i+1}$ for $i \in \{1,2,3,4\}$, where we consider that $C_5 = C_1$, that is, $\GAMMA_{P} = \GAMMA$ and $\GAMMA_{C_i} = \{1\}$ for $i \in \{1,2,3,4\}$.
	Moreover, $P^{\gamma^k} = \bigl\{(1/2,1/2)\bigr\}$ for $k \in \mathbb{Z}$.
	Hence 
	\begin{align}
		\chi_{C_i,q} = \chi_\st^{\{1\}}\Ell_{C_i}(q),\quad 
		\Ind^\GAMMA_{\{1\}}\chi_{C_i,q} = \chi_\st^\GAMMA \Ell_{C_i}(q),
	\end{align}
	where $\chi_\st^H$ denote the regular character of a subgroup $H$ of $\GAMMA$, and we obtain
	\begin{align}
		\chi_{P,q} = \Ind^\GAMMA_\GAMMA\chi_{P,q} =
		\begin{cases*}
				\dfrac{1}{20}(\chi_\st^\GAMMA q^2 - 2\chi_\st^\GAMMA q + \chi_\st^\GAMMA)
				 &  $q \equiv 1 \pmod{10}$;\vspace{1mm}\\
				\dfrac{1}{20}(\chi_\st^\GAMMA q^2 - 2\chi_\st^\GAMMA q + 20\boldsymbol{1}) 
				 & $q \equiv 2 \pmod{10}$;\vspace{1mm}\\
				\dfrac{1}{20}(\chi_\st^\GAMMA q^2 - 2\chi_\st^\GAMMA q - 3\chi_\st^\GAMMA)
				 & $q \equiv 3 \pmod{10}$;\vspace{1mm}\\
				\dfrac{1}{20}(\chi_\st^\GAMMA q^2 - 2\chi_\st^\GAMMA q - 8\chi_\st^\GAMMA + 20\boldsymbol{1})
				 & $q \equiv 4 \pmod{10}$;\vspace{1mm}\\
				\dfrac{1}{20}(\chi_\st^\GAMMA q^2 - 2\chi_\st^\GAMMA q + 5\chi_\st^\GAMMA)
				 & $q \equiv 5 \pmod{10}$;\vspace{1mm}\\
				\dfrac{1}{20}(\chi_\st^\GAMMA q^2 - 2\chi_\st^\GAMMA q - 4\chi_\st^\GAMMA + 20\boldsymbol{1})
				 & $q \equiv 6 \pmod{10}$;\vspace{1mm}\\
				\dfrac{1}{20}(\chi_\st^\GAMMA q^2 - 2\chi_\st^\GAMMA q + 5 \chi_\st^\GAMMA)
				 & $q \equiv 7 \pmod{10}$;\vspace{1mm}\\
				\dfrac{1}{20}(\chi_\st^\GAMMA q^2 - 2\chi_\st^\GAMMA q - 8\chi_\st^\GAMMA + 20\boldsymbol{1})
				 & $q \equiv 8 \pmod{10}$;\vspace{1mm}\\
				\dfrac{1}{20}(\chi_\st^\GAMMA q^2 - 2\chi_\st^\GAMMA q - 3\chi_\st^\GAMMA)
				 & $q \equiv 9 \pmod{10}$;\vspace{1mm}\\
				\dfrac{1}{20}(\chi_\st^\GAMMA q^2 - 2\chi_\st^\GAMMA q + 20\boldsymbol{1})
				 & $q \equiv 10 \pmod{10}$.
		\end{cases*}
	\end{align}
	
	These computations imply that 
	\begin{align}
		\chi_{\A,q} &= \Ind^\GAMMA_\GAMMA\chi_{P,q} + \Ind^\GAMMA_{\{1\}}\chi_{C_i,q}\\
		 &= \begin{cases*}
			\dfrac{1}{4}(\chi_\st q^2 - 2\chi_\st q + \chi_\st) & $\gcd\{10,q\} = 1$;\vspace{1mm}\\
			\dfrac{1}{4}(\chi_\st q^2 - 2\chi_\st q + 4\boldsymbol{1}) & $\gcd\{10,q\} = 2$;\vspace{1mm}\\
			\dfrac{1}{4}(\chi_\st q^2 - 2\chi_\st q + 5\chi_\st) & $\gcd\{10,q\} = 5$;\vspace{1mm}\\
			\dfrac{1}{4}(\chi_\st q^2 - 2\chi_\st q + 4\chi_\st + 4\boldsymbol{1}) & $\gcd\{10,q\} = 10$.
		\end{cases*}
	\end{align}
	
\end{example}

\section{Coxeter arrangements  with the Weyl group action}\label{S5}

For a more specific example, we compute $\chi_{\A,q}$ for the Coxeter arrangements.

\subsection{Coxeter arrangements and Weyl groups}\label{sec5.1}
\ \\*
For details of the root systems, the Coxeter arrangements and the Weyl groups, refer to \cite{Bourbaki,Humphreys}.
Let $E = \mathbb{R}^\ell$ be the Euclidean space with inner product $(\cdot,\cdot)$.
Let $\PHI \subseteq E$ be an irreducible root system.
Fix the positive roots $\PHI^+$ and the simple roots $\{\alpha_1,\ldots,\alpha_\ell\} \subseteq \PHI^+$.
Let $\tilde{\alpha}\in \PHI^+$ be the highest root, and it can be expressed by
\begin{align}
	\tilde{\alpha} = c_1\alpha_1 + \cdots + c_\ell\alpha_\ell.
\end{align}
The \textit{root lattice} $Q \ceq Q(\PHI)$ is the lattice generated by $\PHI$.
The \textit{coweight lattice} $Z \ceq Z(\PHI)$ and \textit{coroot lattice} $\veeQ \ceq \veeQ(\PHI)$ are lattices defined by 
\begin{align}
	Z &= Z(\PHI) = \bigset{x \in E}{(\alpha,x) \in \mathbb{Z} \tforall \alpha \in \PHI},\\
	\veeQ &= \veeQ(\PHI) = \sum_{\alpha \in \PHI}\mathbb{Z}\dfrac{2\alpha}{(\alpha,\alpha)}.
\end{align}
Then the coweight lattice $Z$ can be regarded as the dual lattice of the root lattice $Q$.
Thus we can consider each element $\beta \in Q$ to be a element of $Z^\vee$ and a function $\beta : E \lra \mathbb{R}$ defined by 
\begin{align}
	\alpha(x) = (\alpha,x).
\end{align} 
Furthermore, the coroot lattice $\veeQ$ is a subgroup of the coweight lattice $Z$ with a finite index $f \ceq (Z:\veeQ)$ (called an \textit{index of connection} in Lie theory).
It is known in \cite[\S4.9, Theorem]{Humphreys} that 
\begin{align}
	\#W = f \cdot \ell! \cdot c_1 \cdots c_\ell.
\end{align}
Let $\{\varpivee_1,\ldots,\varpivee_\ell\} \subseteq Z$ be the dual basis to the simple roots $\{\alpha_1,\ldots,\alpha_\ell\}$, that is, 
\begin{align}
	(\alpha_i,\varpivee_j) = \delta_{ij}.
\end{align}

Let $W \ceq W(\PHI)$ be the Weyl group of $\PHI$ generated by the reflections $s_{\alpha} : E \lra E$ with respect to $\alpha \in \PHI$:
\begin{align}
	s_{\alpha}(x) = x - \dfrac{2(\alpha,x)}{(\alpha,\alpha)}\alpha.
\end{align}
The root system $\PHI$ is invariant under $W$.
Then the Weyl group $W$ acts on $E$.
Thus $W$-action $\rho_Q : W \lra \GL(Q)$ and $\rho_Z : W \lra \GL(Z)$ are induced, and they satisfy
\begin{align}
	\rho_Q(w) : \beta \lmapsto w(\beta),\quad \rho_Z(w) : x \lmapsto w(x)
\end{align}
for $w \in W$, $\beta \in Q$ and $x \in Z$.
Moreover, the inner product $(\cdot,\cdot)$ is invariant under $W$.
Therefore we have $\rho_Q = \rho_Z^\vee$ since
\begin{align}
	\rho_Q(\beta)(x) = (\rho_Q(\beta),x) = (w(\beta),x) = (\beta,w^{-1}(x)) = (\beta,\rho_Z(w^{-1})(x)) = \rho_Z^\vee(\beta)(x)
\end{align}
for $\beta \in Q$ and $x \in Z$.

Let $\A \ceq \A(\PHI)$ be the Coxeter arrangement, which is the collection of the reflecting hyperplanes for all reflections in $W$, that is, 
\begin{align}
	\A = \A(\PHI) = \bigset{H_{\alpha}}{\alpha \in \PHI^+},\qquad H_{\alpha} = H_{-\alpha} = \bigset{x \in E}{(\alpha,x) = 0}\quad (\alpha \in \PHI^+).
\end{align}
Then $\A$ is $W$-invariant since
\begin{align}
	wH_\alpha &= \bigset{w(x) \in E}{(\alpha,x) = 0}\\
	 &= \bigset{x \in E}{(\alpha,\rho_Z(w^{-1})(x)) = 0}\\
	 &= \bigset{x \in E}{(\rho_Q(w)(\alpha),x) = 0}\\
	 &= H_{\rho_Q(w)(\alpha)}.
\end{align}

\subsection{Fundamental alcoves}\label{sec5.2}
\ \\*
The \textit{fundamental alcove} $A_\circ$ is the chamber of $\A^\aff$ defined by
\begin{align}
	A_\circ  = \Bigset{x \in E}{\begin{lgathered}
			(\alpha_i,x) > 0 \tforall i \in \{1,\ldots,\ell-1\},\\
			(\tilde{\alpha},x) < 1
		\end{lgathered}
	}
	= H_{\alpha_1}^{0,+} \cap  \cdots \cap H_{\alpha_{\ell}}^{0,+} \cap H_{\tilde{\alpha}}^{1,-}
\end{align}
(a chamber of $\A^\aff$ is often called an \textit{alcove}).
The closure $\overline{A_\circ}$ is the convex hull
\begin{align}
	\conv\left\{0,\, \frac{\varpivee_1}{c_1},\, \ldots,\, \frac{\varpivee_{\ell}}{c_\ell}\right\},
\end{align}
where $c_1,\ldots,c_\ell$ are the coefficients of the highest root $\tilde{\alpha}$.
Let $P^\diamond$ denote the fundamental domain of $Z$ defined by
\begin{align}
	P^\diamond = \sum_{i=1}^{\ell}(0,1]\varpivee_i = \Bigset{\sum_{i=1}^{\ell}k_i\varpivee_i \in E}{0 < k_i \leq 1 \tforall i \in \{1,\ldots,\ell\}},
\end{align}
and let $\mathcal{C}^\diamond$ denote the set of chambers of $\A^\aff$ contained in $P^\diamond$.
The fundamemtal alcove $A_\circ$ belongs to $\mathcal{C}^\diamond$.
Moreover, it is known that $\#\mathcal{C}^\diamond = \#W / f = \ell! \cdot c_1 \cdots c_\ell$ (\cite[\S4.9]{Humphreys} and see also \cite[\S2.3]{YoshinagaW}).
Let $T \ceq E/Z$ be the $\ell$-torus and let $\pi_T : E \lra T$ be the natural projection.
Then 
\begin{align}
	T = \bigset{\pi_T(x)}{x \in P^\diamond},\qquad \mathcal{C}_T = \bigset{\pi_T(C)}{C \in \mathcal{C}^\diamond}.
\end{align}

The Weyl group $W$ acts on $T$ via $\rho_T : W \lra \GL(T)$ satisfying
\begin{align}
	\rho_T(w) \circ \pi_T = \pi_T \circ \rho_Z(w).
\end{align}
To compute $\chi_{\A,q}$, we consider the $W$-orbits of $\mathcal{C}_T$.
Let $W(A_\circ)$ denote the $W$-orbit of $\pi_T(A_\circ)$.

\begin{proposition}
	The $W$-orbit $W(A_\circ)$ is equal to $\mathcal{C}_T$.
\end{proposition}
\begin{proof}
	Let $W_\aff$ be the affine Weyl group, which is the semidirect product of $W$ and the translation group corresponding to the coroot lattice $\veeQ$.
	Then the closure $\overline{A_\circ}$ is a fundamental domain of $W_\aff$.
	Furthermore, $W_\aff$ acts simply transitively on the set $\mathcal{C}(\A^\aff)$ of chambers of $\A$ (see \cite[\S4]{Humphreys}).
	
	Let $C \in \mathcal{C}_T$.
	Then $\pi_T(C) \in \mathcal{C}^\diamond \subseteq \mathcal{C}(\A^\aff)$.
	Simple transitivity of $W$-action on $\mathcal{C}(A^\aff)$ implies that there exists $w \in W_\aff$ such that $w(A_\circ) = C$.
	Hence there exist $w_0 \in W$ and $t \in \veeQ$ such that $w_0(A_\circ + t) = C$.
	Since $\veeQ \subseteq Z$, we have $\pi_T(A_\circ) = \pi_T(A_\circ + t)$.
	Therefore $\pi_T(C) = \rho_T(w_0)\bigl(\pi_T(A_\circ)\bigr) \in W(A_\circ)$.
\end{proof}

Let $W_{A_\circ}$ denote the isotropy group of $\pi_T(A_\circ)$:
\begin{align}
	W_{A_\circ} = \bigset{w \in W}{\rho_T(w)\bigl(\pi_T(A_\circ)\bigr) = \pi_T(A_\circ)}.
\end{align}
Note that $\#W_{A_\circ} = f$ since $\#W(A_\circ) = \#\mathcal{C}_T = \#\mathcal{C}^\diamond$.
Using \cref{Main result 2}, we have the following result.

\begin{theorem}[Restatement of \cref{thm1.3}]\label{weylver}
	For $q \in \mathbb{Z}_{>0}$, we have
	\begin{align}
		\chi_{\A,q} = \Ind^{W}_{W_{{A_\circ}}}\chi_{A_\circ,q}. \label{mt3}
	\end{align}
\end{theorem}

In particular, by substituting $1$ into \cref{mt3}, we have
\begin{align}
	\chi_\A^\quasi(q) &= \dfrac{\#W}{\#W_{{A_\circ}}}\Ell_{A_\circ}(q) = \dfrac{\#W}{f}(-1)^\ell\Ell_{\overline{A_\circ}}(-q). \label{cqpeqp}
\end{align}
The above equation had already been obtained in \cite[Proposition 3.7]{YoshinagaW} (see also \cite{KamiyaTakemuraTeraoR,Suter}).
Hence the equation \cref{mt3} can be considered to be an equivariant version of \cref{cqpeqp}.

\subsection{Case of type $A_\ell$}\label{sec5.3}
\ \\*
Let $\{e_1,\ldots,e_{\ell+1}\}$ be the standard basis for the Euclidean space $\mathbb{R}^{\ell+1}$.
Define $E$ as the $\ell$-dimensional subspace of $\mathbb{R}^{\ell+1}$ consisting the points the sum of whose coordinates is zero.
The root system of type $A_\ell$ is defined by 
\begin{align}
	\PHI = \PHI(A_\ell) = \bigset{e_i-e_j \in E}{i,j \in \{1,\ldots,\ell+1\},\ i \neq j}.
\end{align}
Let $\alpha_i \ceq e_i - e_{i+1}$ for $i \in \{1,\ldots,\ell\}$.
Then $\{\alpha_1,\ldots,\alpha_\ell\}$ forms simple roots of $\PHI$, and its dual basis $\{\varpivee_1,\ldots,\varpivee_{\ell}\}$ is expressed as 
\begin{align}
	\varpivee_j = (e_1 + \cdots + e_j) - \dfrac{j}{\ell+1}e_\all,
\end{align}
where $e_\all = e_1 + \cdots + e_{\ell+1}$.
We suppose that $\varpivee_0 \ceq 0$.
The Weyl group of type $A_\ell$ is the symmetric group $\mathfrak{S}_{\ell+1}$ of degree $\ell+1$.
The group $\mathfrak{S}_{\ell+1}$ acts on $Q$ and $Z$ as the permutations of the standard basis $\{e_1,\ldots,e_{\ell+1}\}$, that is,
\begin{align}
	\sigma : e_i \lmapsto e_{\sigma(i)}
\end{align}
for $\sigma \in \mathfrak{S}_{\ell+1}$ and $i \in \{1,\ldots,\ell+1\}$.
For example, let $\sigma \ceq (1\ 2\ 3\ 4) \in \mathfrak{S}_{4}$ be a cyclic permutation.
Then 
\begin{align}
	\rho_Q(\sigma)(\alpha_1) = e_2 - e_3 = \alpha_2,\qquad
	\rho_Z(\sigma)(\varpivee_1) = e_2 - \dfrac{1}{4}e_\all = \varpivee_2 - \varpivee_1.
\end{align}

The Coxeter arrangement of type $A_\ell$ is 
\begin{align}
	\A = \A(A_\ell) = \bigset{H_{ij}}{1 \leq i < j \leq \ell+1},
\end{align}
where 
\begin{align}
	H_{ij} = \bigset{x \in E}{(e_i-e_j,\, x) = 0}.
\end{align}

Let $A_\circ$ is the fundamental alcove of type $A_\ell$.
Its closure is the convex hull
\begin{align}
	\overline{A_\circ} = \conv\bigl\{\varpivee_0,\varpivee_1,\ldots,\varpivee_\ell\bigr\}
\end{align}
and is a regular $\ell$-simplex.

\begin{proposition}
	The isotropy group $(\mathfrak{S}_{\ell+1})_{{A_\circ}}$ of $\pi_T(A_\circ)$ is a cyclic group generated by the cyclic permutation $(1\ 2\ \cdots\ \ell+1)$.
\end{proposition}
\begin{proof}
	Let $c$ be the centroid of $A_\circ$.
	Then
	\begin{align}
		c = \dfrac{\varpivee_1 + \cdots + \varpivee_{\ell}}{\ell+1} = \sum_{i=1}^{\ell+1}\dfrac{\ell+1-i}{\ell+1}e_i - \dfrac{\ell}{2(\ell+1)}e_\all.
	\end{align}
	Note that $\rho_T(\sigma)\bigl(\pi_T(A_\circ)\bigr) = \pi_T(A_\circ)$ if and only if $\rho_T(\sigma)\bigl(\pi_T(c)\bigr) = \pi_T(c)$.
	For $\sigma \in \mathfrak{S}_{\ell+1}$, 
	\begin{align}
		\rho_Z(\sigma)(c) = \sum_{i=1}^{\ell+1}\dfrac{\ell+1-i}{\ell+1}e_{\sigma(i)} - \dfrac{\ell}{2(\ell+1)}e_\all.
	\end{align}
	Give $z = z_1\varpivee_1 + \cdots + z_\ell\varpivee_\ell \in Z$ satisfying $\rho_Z(\sigma)(c) \in P^\diamond - z$, where
	\begin{align}
		P^\diamond - z = \bigset{x-z \in E'}{x \in P^\diamond}.
	\end{align}
	Since $\rho_Z(\sigma)$ is a transformation fixing $0$, we can see that $z_1,\ldots,z_{\ell} \in \{0,1\}$.
	
	Suppose that $\rho_T(\sigma)\bigl(\pi_T(c)\bigr) = \pi_T(c)$.
	Then $c = \rho_Z(\sigma)(c) + z$.
	Let $i_j \ceq \sigma^{-1}(j)$ for $j \in \{1,\ldots,\ell+1\}$.
	Since 
	\begin{align}
		c &= \dfrac{\ell}{\ell+1}e_1 + \cdots + 0e_{\ell+1} -\dfrac{\ell}{2(\ell+1)}e_\all,\\ 
		\rho_Z(\sigma)(c) &= \dfrac{\ell+1 - i_1}{\ell+1}e_1 + \cdots + \dfrac{\ell+1 - i_{\ell+1}}{\ell+1}e_{\ell+1} - \dfrac{\ell}{2(\ell+1)}e_\all,
	\end{align}
	by comparing the coefficients of $e_1$ and $e_{\ell+1}$, respectively, we obtain 
	\begin{align}
		\sum_{z_i = 1}i = \ell+1 - i_{\ell+1},\qquad \sum_{z_i = 1}(\ell+1-i) = i_1 - 1, \label{dii}
	\end{align}
	and hence 
	\begin{align}
		i_1 - i_{\ell+1} = \#\bigset{i}{z_i = 1} \cdot (\ell+1) - (\ell+1) + 1.
	\end{align}
	Therefore we have $\#\Bigset{i}{z_i = 1} \leq 1$.
	When $\#\Bigset{i}{z_i = 1} =0$, then $\sigma = 1$ since $\rho_Z(\sigma)(c) = c$.
	Suppose that $\#\Bigset{i}{z_i = 1} = 1$.
	The left-side equation of \cref{dii} implies that $z_{\ell+1 - i_{\ell+1}} = 1$, that is, $z = \varpivee_{\ell+1-i_{\ell+1}}$.
	Then
	\begin{align}
		\rho_Z(\sigma)(c) &= c - \varpivee_{\ell+1-i_{\ell+1}}\\
		&= \dfrac{\ell-i_{\ell+1}}{\ell+1}e_1 + \cdots + \dfrac{1}{\ell+1}e_{\ell-i_{\ell+1}} + 0e_{\ell+1-i_{\ell+1}} + \dfrac{\ell}{\ell+1}e_{\ell+2-i_{\ell+1}} + \cdots + \dfrac{\ell+1-i_{\ell+1}}{\ell+1}e_{\ell+1} - \dfrac{\ell}{2(\ell+1)}e_\all.
	\end{align}
	Thus $\sigma = (1\ 2\ \cdots \ \ell+1)^{i_{\ell+1}}$, and hence $(\mathfrak{S}_{\ell+1})_{{A_\circ}} = \bigl\langle(1\ 2\ \cdots\ \ell+1)\bigr\rangle$.
\end{proof}

Fix a permutation $\sigma \ceq (1\ 2\ \cdots \ \ell+1)$ and an integer $k \in \{1,\ldots,\ell+1\}$.
Then we have
\begin{align}
	\rho_Z(\sigma^k)(\varpivee_j) = \varpivee_{j+k} - \varpivee_k \quad (j \in \{0,\ldots,\ell\}),
	 \label{fcp}
\end{align}
where we consider that $\varpivee_t = \varpivee_j$ for $0 \leq j < \ell+1 \leq t$ with $j \equiv t \pmod{\ell+1}$. 
For example, let $\ell = 3$.
Then
\begin{align}
	\rho_Z(\sigma^2)(\varpivee_3) = \varpivee_5 - \varpivee_2 = \varpivee_1 - \varpivee_2,\qquad \rho_Z(\sigma^3)(\varpivee_1) = \varpivee_4 - \varpivee_3 = -\varpivee_3.
\end{align}

The \textit{cycle type} of $\tau \in \mathfrak{S}_{\ell+1}$ is the decreasing sequence of the length of the cycles in the cycle decomposition of $\tau$, which is an integer partition of $\ell+1$.
Two permutations of $\mathfrak{S}_{\ell+1}$ are conjugate if and only if they have the same cycle type.

The permutation $\sigma^k = (1\ 2\ \cdots\ \ell+1)^k$ can be decomposed into $g$ cyclic permutations of length $d$, that is, $\sigma^k$ has the cycle type $(d,\ldots,d)$, where
\begin{align}
	g = \gcd\{\ell+1,\, k\},\qquad d = \dfrac{\ell+1}{g}.
\end{align}

For each $b \in \{0,\ldots,g-1\}$, let
\begin{align}
	J_{g,b} 
	\ceq \bigset{j \in \{0,\ldots,\ell\}}{j \equiv b \pmod{g}},\qquad 
\end{align}
and define 
\begin{align}
	\bar{f}^g_b \ceq \sum_{j \in J_{g,b}}\varpivee_j,\qquad 
	f^g_b \ceq \dfrac{1}{d} \cdot \bar{f}^g_b.
\end{align}
Then $f^g_b \in \overline{A_\circ}$ since $\#J_{g,b} = d$.
For example, let $\ell = 3$.
Then 
\begin{align}
	f^1_0 = \dfrac{\varpivee_1 + \varpivee_2 + \varpivee_3}{4},\\
	f^2_0 = \dfrac{\varpivee_2}{2},\quad f^2_1 = \dfrac{\varpivee_1 + \varpivee_3}{2},\\
	f^4_0 = 0,\quad f^4_1 = \varpivee_1,\quad f^4_2 = \varpivee_2,\quad f^4_3 = \varpivee_3.
\end{align}

\begin{proposition}
	Let $(\overline{A_\circ})^{\sigma^k}$ be the set of points in  $\pi_T(\overline{A_\circ})$ fixed by $\rho_T(\sigma^k)$.
	Then we have
	\begin{align}
		(\overline{A_\circ})^{\sigma^k} &= \conv\bigset{f^g_b}{g \in \{0,\ldots,g-1\}}.\label{fconv}
	\end{align}
\end{proposition}
\begin{proof}
	Since $j + k \in J_{g,b}$ if $j \in J_{g,b}$, the formula \cref{fcp} implies that 
	\begin{align}
		\rho_T(\sigma^k)\bigl(\pi_T(f^g_b)\bigr) 
		&= \pi_T \circ \rho(\sigma^k)\left(\dfrac{1}{d}\sum_{j \in J_{g,b}}\varpivee_j \right)\\
		&= \pi_T\left(\dfrac{1}{d}\sum_{j \in J_{g,k}}(\varpivee_{j+k} - \varpivee_k)\right)\\
		&= \pi_T(f^g_b - \varpivee_k)\\
		&= \pi_T(f^g_b).
	\end{align}
	Hence $f^g_b \in (\overline{A_\circ})^{\sigma^k}$ for all $b \in \{0,\ldots,g-1\}$.
	By the linearity of $W$-action, we have
	\begin{align}
		\conv\bigset{f^g_b}{g \in \{0,\ldots,g-1\}} \subseteq (\overline{A_\circ})^{\sigma^k}.
	\end{align}
	
	Let $\pi_T(x) \in (\overline{A_\circ})^{\sigma^k}$ with $x = x_0\varpivee_0 + x_1\varpivee_1 + \cdots + x_\ell\varpivee_\ell \in \overline{A_\circ}$ ($x_0,\ldots,x_\ell$ satisfy $x_0 + \cdots + x_\ell = 1$).
	Since
	\begin{align}
		\rho_Z(\sigma^k)(x) 
		&= x_0(\varpivee_k - \varpivee_k) + x_1(\varpivee_{1+k} -  \varpivee_k) + \cdots + x_\ell(\varpivee_{\ell+k} - \varpivee_k)\\
		&= x_0\varpivee_k + x_1\varpivee_{1+k} + \cdots + x_\ell\varpivee_{\ell+k} - \varpivee_k,
	\end{align}
	then
	\begin{align}
		\rho_T(\sigma^k)\bigl(\pi_T(x)\bigr) = \pi_T(x_0\varpivee_k + x_1\varpivee_{1+k} + \cdots + x_\ell\varpivee_{\ell+k}).
	\end{align}
	Therefore we have $x_j = x_{j+k}$ for all $j \in \{0,\ldots,\ell\}$, where we consider that $x_t = x_j$ for $0 \leq j < \ell+1 \leq t$ with $j \equiv t \pmod{\ell+1}$.
	Hence it implies that $x$ belongs to the right-hand set of \cref{fconv}.
\end{proof}

For $q \in \mathbb{Z}_{>0}$, the equation \cref{fconv} implies that
\begin{align}
	q \cdot (\overline{A_\circ})^{\sigma^k} 
	&= \bigset{x_0f^g_0 + \cdots + x_{g-1}f^g_{g-1}}{x_b \geq 0,\ x_0 + \cdots + x_{g-1} = q}\\
	&= \bigset{\dfrac{x_0}{d}\bar{f}^g_0 + \cdots + \dfrac{x_{g-1}}{d}\bar{f}^g_{g-1}}{x_b \geq 0,\ x_0 + \cdots + x_{g-1} = q}.\label{qpt}
\end{align}

\begin{proposition}
	The Ehrhart quasi-polynomial of $(\overline{A_\circ})^{\sigma^k}$ is 
	\begin{align}
		\Ell_{(\overline{A_\circ})^{\sigma^k}}(q) = \begin{cases*}
			0 & if $q \not\in d\mathbb{Z}$;\\
			\displaystyle\binom{g-1+\frac{q}{d}}{g-1} & if $q \in d\mathbb{Z}$.
		\end{cases*}
	\end{align}
	Hence 
	\begin{align}
		\Ell_{A_\circ^{\sigma^k}}(q) = (-1)^{g-1}\Ell_{(\overline{A_\circ})^{\sigma^k}}(-q) &= \begin{cases*}
			0 & if $q \not\in d\mathbb{Z}$;\\
			\displaystyle(-1)^{g-1}\binom{g-1-\frac{q}{d}}{g-1} & if $q \in d\mathbb{Z}$
		\end{cases*}\\
		 &= \begin{cases*}
			0 & if $q \not\in d\mathbb{Z}$;\\
			\dfrac{(q-d)(q-2d) \cdots (q-d(g-1))}{d^{g-1}(g-1)!} & if $q \in d\mathbb{Z}$.
		\end{cases*}
	\end{align}
\end{proposition}
\begin{proof}
	The equation \cref{qpt} implies that $x_0f^g_0 + \cdots + x_{g-1}f^g_{g-1} \in q \cdot (\overline{A_\circ})^{\sigma^k}$ belongs to $Z$ if and only if  $x_0,\ldots,x_{g-1} \in d\mathbb{Z}$.
	Therefore $q \cdot  (\overline{A_\circ})^{\sigma^k} \cap Z = \emptyset$ if $q \not\in d\mathbb{Z}$.
	
	Suppose that $q \in d\mathbb{Z}$.
	Then we have
	\begin{align}
		\Ell_{(\overline{A_\circ})^{\sigma^k}}(q) 
		&= \#\bigl(q \cdot (\overline{A_\circ})^{\sigma^k} \cap Z\bigr)\\
		&= \#\bigset{\dfrac{x_0}{d}\bar{f}^g_0 + \cdots + \dfrac{x_{g-1}}{d}\bar{f}^g_{g-1}}{x_b \in d\mathbb{Z}_{>0},\ x_0 + \cdots + x_{g-1} = q}\\
		&= \#\bigset{(x_0,\ldots,x_{g-1}) \in (\mathbb{Z}_{>0})^g}{x_0 + \cdots + x_{g-1} = \dfrac{q}{d}}\\
		&= \binom{g-1+\frac{q}{d}}{g-1}.
	\end{align}
\end{proof}

Define
\begin{align}
	\varphi_g(\ell+1) = \#\bigset{i \in \{1,\ldots,\ell+1\}}{\gcd\{\ell+1,\, i\} = g}.
\end{align}
When $g=1$, then $\varphi_g$ is known as the Euler's totient function.
The value $\varphi_g(\ell+1)$ is equal to the number of permutations in $(\mathfrak{S}_{\ell+1})_{{A_\circ}} = \langle \sigma\rangle$ with the cycle type $(d,\ldots,d)$.

\begin{proposition}
	Let $\sigma \ceq (1\ 2\ \cdots\ \ell+1) \in \mathfrak{S}_{\ell+1}$ and $k \in \{1,\ldots,\ell+1\}$.
	Then we have
	\begin{align}
		\#\bigset{\tau \in \mathfrak{S}_{\ell+1}}{\tau\sigma^k\tau^{-1} \in \langle \sigma\rangle} = \varphi_{g}(\ell+1) \cdot (\ell+1)(\ell+1-d)(\ell+1-2d) \cdots (\ell+1-d(g-1)). \label{nofc}
	\end{align}
\end{proposition}
\begin{proof}
	Give the cycle decomposition 
	\begin{align}
		\sigma^k = (i_{11}\ \cdots\ i_{1d})(i_{21}\ \cdots\ i_{2d}) \cdots\cdots (i_{g,1}\ \cdots\ i_{g,d}).
	\end{align}
	Then it is well known that 
	\begin{align}
		\tau\sigma^k\tau^{-1} = \bigl(\tau(i_{11})\ \cdots\ \tau(i_{1d})\bigr)\bigl(\tau(i_{21})\ \cdots\ \tau(i_{2d})\bigr) \cdots\cdots \bigl(\tau(i_{g,1})\ \cdots\ \tau(i_{g,d})\bigr)
	\end{align}
	for $\tau \in \mathfrak{S}_{\ell+1}$.
	If $\tau\sigma^k\tau^{-1} \in \langle\sigma\rangle$, then $\tau\sigma^k\tau^{-1}$ is equal to one of the $\varphi_{g}(\ell+1)$ permutations with the same cycle type contained in $\langle\sigma\rangle$.
	Since the order of cycles in the cycle decomposition and the order of integers in each cycle do not affect the resulting permutation, we  can choose an arbitrary number from $\{1,\ldots,\ell+1\}$ as $\tau(i_{11})$.
	For the same reason, we can choose any number not included in the cycle containing $\tau(i_{11})$ as $\tau(i_{21})$.
	Thus we obtain the formula \cref{nofc}.
\end{proof}

Therefore we have
\begin{align}
	&\left(\Ind^{\mathfrak{S}_{\ell+1}}_{(\mathfrak{S}_{\ell+1})_{{A_\circ}}}\chi_{A_\circ,q}\right)(\sigma^k) \\
	&\qquad \qquad = \dfrac{1}{\#\langle\sigma\rangle}\sum_{\substack{\tau \in \GAMMA\\\tau\sigma^k\tau^{-1} \in \langle\sigma\rangle}}\chi_{A_\circ,q}(\tau\sigma^k\tau^{-1})\\
	&\qquad \qquad = \dfrac{\#\bigset{\tau \in \mathfrak{S}_{\ell+1}}{\tau\sigma^k\tau^{-1} \in \langle \sigma\rangle}}{\ell+1} \cdot \Ell_{A_\circ^{\sigma^k}}(q)\\
	&\qquad \qquad = \begin{cases*}
		0 & if $q \not\in d \mathbb{Z}$;\\
		\dfrac{\varphi_{g}(\ell+1) \cdot (\ell+1-d) \cdots (\ell+1-d(g-1)) \cdot (q-d) \cdots (q-d(g-1))}{d^{g-1}(g-1)!} & if $q \in d\mathbb{Z}$
	\end{cases*}\\
	&\qquad \qquad = \begin{cases*}
		0 & if $q \not\in d \mathbb{Z}$;\\
		\varphi_{g}(\ell+1) \cdot (q-d) \cdots (q-d(g-1)) & if $q \in d\mathbb{Z}$.
	\end{cases*}
\end{align}
By \cref{weylver}, we can obtain the following.

\begin{theorem}
	For each $\sigma \in \mathfrak{S}_{\ell+1}$, we have
	\begin{align}
		\chi_{\A,q}(\sigma) = \begin{cases*}
			\varphi_{g}(\ell+1) \cdot (q-d) \cdots (q-d(g-1)) & if $q \in d\mathbb{Z}$ and $\sigma$ has the cycle type $(d,\ldots,d)$;\\
		0 & otherwise,
	\end{cases*}
	\end{align}
	where $g$ and $d$ are divisors of $\ell+1$ satisfying $gd = \ell+1$.
	Hence $\chi_{\A,q}$ has the minimum period $\ell+1$.
\end{theorem}

\section*{Acknowledgement}

The author would like to thank Professor Masahiko Yoshinaga for the
helpful discussions and comments on this research.
The author also acknowledge support by JSPS KAKENHI, Grant Number 25KJ1735.


\bibliographystyle{amsplain}
\bibliography{bibfile}

\end{document}